\documentclass[11pt,preprint]{imsart}
\usepackage{amsmath,amsthm,verbatim,amssymb,amsfonts,amscd, graphicx}
\usepackage{amssymb,bm}
\usepackage{mathrsfs}
\usepackage{hyperref}
\usepackage{graphicx}
\hypersetup{
    colorlinks,
    citecolor=black,
    filecolor=black,
    linkcolor=blue,
    urlcolor=black
}
\usepackage{url}
\usepackage{verbatim}
\usepackage{constants}
\usepackage{enumerate}
\newcommand{\beas}{\begin{eqnarray*}}
\newcommand{\enas}{\end{eqnarray*}}
\newcommand{\bea}{\begin{eqnarray}}
\newcommand{\ena}{\end{eqnarray}}
\newcommand{\bms}{\begin{multline*}}
\newcommand{\ems}{\end{multline*}}
\newcommand{\bels}{\begin{align*}}
\newcommand{\enls}{\end{align*}}
\newcommand{\bel}{\begin{align}}
\newcommand{\enl}{\end{align}}

\newcommand{\ignore}[1]{}

\newtheorem{theorem}{Theorem}[section]
\newtheorem{corollary}{Corollary}[section]

\newtheorem{proposition}{Proposition}[section]
\newtheorem{remark}{Remark}[section]
\newtheorem{lemma}{Lemma}[section]
\newtheorem{definition}{Definition}[section]

\newtheorem{assumption}{Assumption}[section]

\def\blfootnote{\xdef\@thefnmark{}\@footnotetext}

\newcommand{\xcolor}[1]{\textcolor{black}{#1}}

\newcommand{\expect}[1]{\mathbb{E}{\l(#1\r)}}

\def\r{\right}
\def\l{\left}

\begin{document}
\begin{frontmatter}
\title{Online Learning in Weakly Coupled Markov Decision Processes: A Convergence Time Study}
\runtitle{Online constrained MDPs}

\begin{aug}
\author{\fnms{Xiaohan} \snm{Wei}\thanksref{t1}\ead[label=e1]{xiaohanw@usc.edu}},
\author{\fnms{Hao} \snm{Yu}\thanksref{t1}\ead[label=e2]{yuhao@usc.edu}}
\and
\author{\fnms{Michael} \snm{J. Neely}\thanksref{t1}\ead[label=e3]{mikejneely@gmail.com}}
\thankstext{t1}{Department of Electrical Engineering, University of Southern California}
\runauthor{X. Wei, H. Yu, M. J. Neely}

\affiliation{University of Southern California}
\end{aug}

\maketitle

\maketitle

\begin{abstract}
We consider multiple parallel Markov decision processes (MDPs) coupled by global constraints, where the time varying objective and constraint functions can only be observed \emph{after}  the decision is made.  Special attention is given to how well the decision maker can perform in $T$ slots, starting from any state, compared to the best feasible randomized stationary policy in hindsight. We develop a new distributed online algorithm where each MDP makes its own decision each slot after observing a multiplier computed from past information. 
While the scenario is significantly more challenging than the classical online learning context, the algorithm is shown to have a tight $O(\sqrt{T})$ regret and constraint violations simultaneously.  To obtain such a bound, we combine several new ingredients including ergodicity and mixing time bound in weakly coupled MDPs, a new regret analysis for online constrained optimization, a drift analysis for queue processes, and a perturbation analysis based on Farkas' Lemma.  
\end{abstract}


\end{frontmatter}

\section{Introduction}\label{sec:intro}

This paper considers online constrained Markov decision processes (OCMDP) where both the objective and constraint functions can vary each time slot after the decision is made.  We assume a slotted time scenario with time slots $t \in \{0, 1, 2, \ldots\}$.  The OCMDP consists of $K$ parallel Markov decision processes with indices $k \in \{1, 2, \ldots, K\}$.  The $k$-th MDP has state space $\mathcal{S}^{(k)}$, action space $\mathcal{A}^{(k)}$, and transition probability matrix $P_a^{(k)}$ which depends on the chosen action $a \in \mathcal{A}^{(k)}$.  Specifically, $P_a^{(k)} = (P_a^{(k)}(s,s'))$ where 
\[
P_a^{(k)}(s,s')  = Pr\l(s_{t+1}^{(k)}=s'~\l|~s_t^{(k)} = s,~a_t^{(k)}=a\r.\r),
\]
where $s_t^{(k)}$ and $a_t^{(k)}$ are the state and action for system $k$ on slot $t$. 
We assume that both the state space and the action space are finite for all $k\in\{1,2,\cdots,K\}$. 
 
 After each MDP $k \in \{1, \ldots, K\}$ makes the decision at time $t$ (and assuming the current state is $s_t^{(k)} = s$ and the action is $a_t^{(k)}=a)$, the following information is revealed: 
 \begin{enumerate} 
 \item The next state $s_{t+1}^{(k)}$. 
 
 \item A penalty function $f_t^{(k)}(s,a)$ that depends on the current state $s$ and the current action $a$. 
 
 \item A collection of $m$ constraint functions $g_{1,t}^{(k)}(s,a), \ldots, g_{m,t}^{(k)}(s,a)$ that depend on $s$ and $a$. 
 \end{enumerate} 
 The functions $f_t^{(k)}$ and $g_{i,t}^{(k)}$ are all bounded mappings from $\mathcal{S}^{(k)} \times \mathcal{A}^{(k)}$ to $\mathbb{R}$ and represent different types of costs incurred by system $k$ on slot $t$ (depending on the current state and action).  
For example, in a multi-server data center, the different systems $k \in \{1, \ldots, K\}$ can represent different servers, the cost function for a particular server $k$ might represent energy or monetary expenditure for that server, and the constraint costs for server $k$ 
can represent negative rewards such as service rates or qualities.  Coupling between the server systems comes from using all of them to collectively support a common stream of arriving jobs. 

 A key aspect of this general problem is that the functions $f_t^{(k)}$ and $g_{i,t}^{(k)}$ are unknown until after the slot $t$ decision is made. Thus, the precise costs incurred by each system are only known at the end of the slot.  
 For a fixed time horizon of $T$ slots, the overall penalty and constraint accumulation resulting from a policy $\mathscr{P}$ is: 
\begin{equation}\label{main-regret}
F_T(d_0,\mathscr{P}) := \expect{\left.\sum_{t=1}^T\sum_{k=1}^Kf_t^{(k)}\l(a_t^{(k)},s_t^{(k)}\r)\right|~d_0,\mathscr{P}},
\end{equation}
and
\begin{equation*}
G_{i,T}(d_0,\mathscr{P}) := \expect{\left.\sum_{t=1}^T\sum_{k=1}^Kg_{i,t}^{(k)}\l(a_t^{(k)},s_t^{(k)}\r)\right|~d_0,\mathscr{P}}, 
\end{equation*}
where $d_0$ represents a given distribution on the initial joint state vector $(s_0^{(1)}, \cdots, s_0^{(K)})$. Note that $(a_t^{(k)}, s_t^{(k)})$ denotes the state-action pair of the $k$th MDP, which is a pair of random variables determined by $d_0$ and $\mathscr{P}$. Define a constraint set 
\begin{equation}\label{main-constraint}
\mathcal{G}:= \{(\mathscr{P},d_0):~G_{i,T}(d_0,\mathscr{P})\leq 0,~i=1,2,\cdots,m\}.
\end{equation}
Define the regret of a policy $\mathscr{P}$ with respect to a particular joint randomized stationary policy $\Pi$ along with an arbitrary starting state distribution $d_0$ as: 
\[
F_T(d_0,\mathscr{P}) - F_T(d_0,\Pi),
\]
The goal of OCMDP is to choose a policy $\mathscr{P}$  so that both the regret and constraint violations grow sublinearly with respect to $T$, where regret is measured against all feasible joint randomized stationary policies $\Pi$.

\subsection{A motivating example}
As an example, consider a data center with a central controller and $K$ servers (see Fig. \ref{fig:single-system}).  Jobs arrive randomly and are stored in a queue to await service.  The system operates in slotted time $t \in \{0, 1, 2,\ldots\}$ and each server $k \in \{1, \ldots, K\}$ is modeled as a 3-state MDP with states \emph{active}, \emph{idle}, and \emph{setup}: 

\begin{itemize}
\item Active: In this state the  server is available to serve jobs.  Server $k$ incurs a time varying electricity cost on every active slot, regardless of whether or not there are jobs to serve. It has a control option to stay active or transition to the idle state. 

\item Idle: In this state no jobs can be served.  This state has multiple sleep modes as control options, each with different per-slot costs and setup times required for transitioning from idle to active. 

\item Setup: This is a transition state between idle and active.  No jobs can be served and there are no control options. The setup costs and durations are (possibly constant) random variables depending on the preceding chosen sleep mode.   
 \end{itemize}
 The goal is to minimize the overall electricity cost subject to stabilizing the job queue.  In a typical data center scenario, the performance of each server on a given slot is governed by the current electricity price and the service rate under each decision, both of which can be time varying and unknown to the server beforehand.  
This problem is challenging because: 
\begin{itemize}
\item  If one server is currently in a setup state, it has zero service rate and cannot make another decision
until it reaches the active state (which typically takes more than one slot), whereas other active servers can make decisions
during this time. Thus, servers are acting asynchronously.

\item The electricity price exhibits variation across time, location, and utility providers. Its behavior is irregular and can be difficult to predict.  As an example, Fig. \ref{fig:price}  plots the average per 5 minute spot market price (between 05/01/2017 and 05/10/2017) at New York zone CENTRL (\cite{nyiso}).  Servers in different locations can have different price offerings, and this piles up the uncertainty across the whole system.
\end{itemize}

 Despite these difficulties, this problem fits into the formulation of this paper: The electricity price acts as the global penalty function, and stability of the queue can be treated as a global constraint that the expected total number of arrivals is less than the expected service rate.

\begin{figure}[htbp]
   \centering
   \includegraphics[height=1.6in]{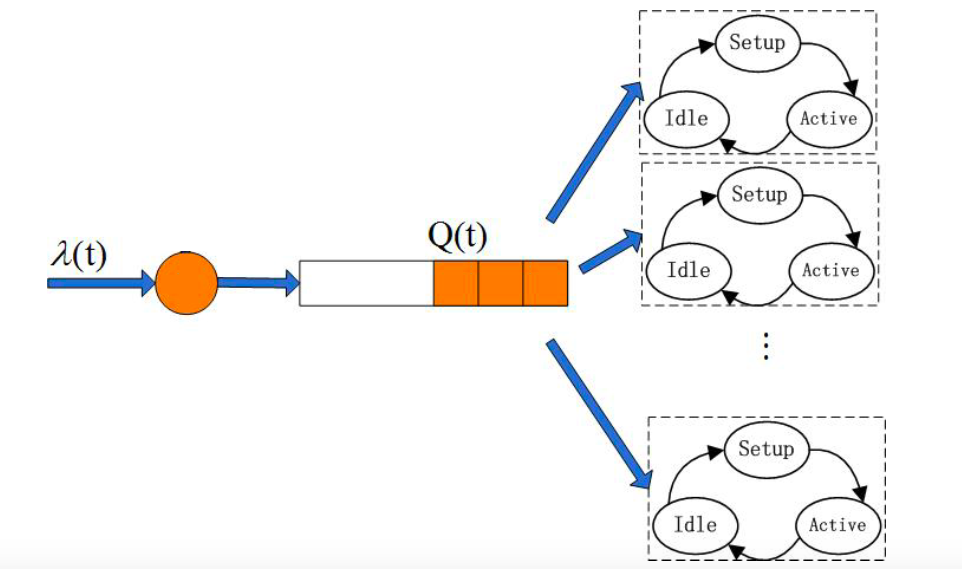} 
   \caption{Illustration of a data center server scheduling model.}
   \label{fig:single-system}
\end{figure}

\begin{figure}[htbp]
   \centering
   \includegraphics[height=2in]{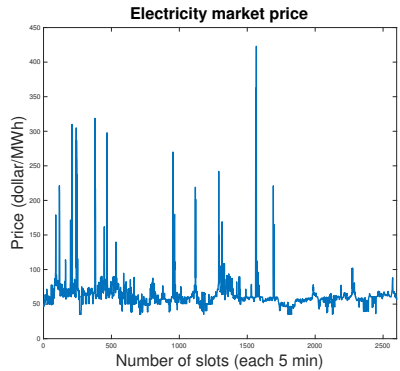} 
   \caption{A typical trace of electricity market price.}
   \label{fig:price}
\end{figure}

A review on data server provision can be found in \cite{gandhi2013dynamic}
and references therein.  Prior data center analysis often assumes the system has up-to-date information on service rates and electricity costs (see, for example, \cite{weidatacenter17},\cite{gandhi2013exact}).   On the other hand, work that treats outdated information (such as \cite{lin2013dynamic}, \cite{urgaonkar2011optimal}) generally does not consider the potential Markov structure of the problem. The current paper treats the Markov structure of the problem and allows rate and price information to be unknown and outdated.

\subsection{Related work}

\begin{itemize}
\item \textbf{Online convex optimization (OCO)}: This concerns multi-round cost minimization with arbitrarily-varying convex loss functions.  Specifically, on each slot $t$ the decision maker chooses decisions $x(t)$ within a convex set $\mathcal{X}$ (before observing the loss function $f^t(x)$) in order to minimize the total \emph{regret} compared to the best fixed decision in hindsight, expressed as: 
\begin{align*}
\text{regret}(T) = \sum_{t=1}^{T} f^t(\mathbf{x}(t))  - \min_{\mathbf{x}\in \mathcal{X}} \sum_{t=1}^T f^t(\mathbf{x}).
\end{align*}
See \cite{hazan2016introduction} for an introduction to OCO.  Zinkevich introduced OCO in \cite{zinkevich2003online}
and shows that an online projection gradient descent (OGD) algorithm achieves $O(\sqrt{T})$ regret. This $O(\sqrt{T})$ regret is proven to be the best in \cite{hazan07ML}, although improved performance is possible if all convex loss functions are \emph{strongly} convex.  The OGD decision requires to compute a projection of a vector onto a set $\mathcal{X}$.  For complicated sets $\mathcal{X}$ with functional equality constraints, e.g., $\mathcal{X} = \{x\in \mathcal{X}_0: g_k (\mathbf{x})\leq 0, k\in\{1,2,\ldots,m\}\}$, the projection can have high complexity. To circumvent the projection, work in \cite{mahdavi2012trading,jenatton2016adaptive,yu2016low,chen2017online}
proposes alternative algorithms with simpler per-slot complexity and that satisfy the inequality constraints in the long term (rather than on every slot).  Recently, new primal-dual type algorithms with low complexity are proposed in \cite{neely2017online,hao2017onlinestochastic} to solve more challenging OCO with time-varying functional inequality constraints.

 \item \textbf{Online Markov decision processes}: 
 This extends OCO to allow systems with a more complex Markov structure. This is similar to the setup of the current paper of minimizing the expression  \eqref{main-regret}, but does not have the constraint set \eqref{main-constraint}.  Unlike traditional OCO, the current penalty depends not only on the current action and the current (unknown) penalty function, but on the current system state (which depends on the history of previous actions).  Further, the number of policies can grow exponentially with the sizes of the state and action spaces, so that solutions can be computationally intensive. The work \cite{even2009online} develops an algorithm in this context with 
  $\mathcal{O}(\sqrt{T})$ regret. Extended algorithms and regularization methods are developed 
  in \cite{yu2009markov}\cite{guan2014online}\cite{dick2014online} to reduce complexity and improve dependencies on the number of states and actions.  Online MDP under bandit feedback (where the decision maker can only observe the penalty corresponding to the chosen action) is considered in \cite{yu2009markov}\cite{neu2010online}.

\item \textbf{Constrained MDPs}: 
This aims to solve classical MDP problems with \emph{known} cost functions but subject to additional constraints on the budget or resources. Linear programming methods for MDPs are found, for example, in  \cite{altman1999constrained}, and algorithms beyond LP are found in \cite{neely2011online} \cite{caramanis2014efficient}.  Formulations closest to our setup appear in recent work on weakly coupled MDPs in \cite{boutilier2016budget}\cite{wei2016theory} that have known cost and resource functions.

\item \textbf{Reinforcement Learning (RL)}: This concerns MDPs with some unknown parameters (such as unknown functions and transition probabilities).  Typically, RL makes stronger assumptions than the online setting, such as an environment that is unknown but fixed, whereas the unknown environment in the online context can change over time. Methods for RL are developed in \cite{bertsekas1995dynamic}\cite{sutton1998reinforcement}\cite{lattimore2013sample}\cite{chen2016stochastic}.
\end{itemize}

\subsection{Our contributions}

The current paper proposes a new framework for online MDPs with time varying constraints.  Further, it considers multiple MDP systems that are weakly coupled.  While the scenario is significantly more challenging than the original Zinkevich OGD context as well as other classical online learning scenarios, the algorithm is shown to achieve  tight $O(\sqrt{T})$ regret in both the objective function and the constraints, which ties the optimal $O(\sqrt{T})$ regret for those simpler unconstrained OCO problems. Along the way, we show the bound grows polynomially with the number of MDPs and linearly with respect to the number of states and actions in each MDP (Theorem \ref{thm:final-regret}). 

\section{Preliminaries}\label{sec:assumption}
\subsection{Basic Definitions}
Throughout this paper, given an MDP with state space $\mathcal{S}$ and action space $\mathcal{A}$,
 a \textit{policy} $\mathscr{P}$ defines a (possibly probabilistic) 
 method of choosing actions $a\in\mathcal{A}$ at state $s\in\mathcal{S}$ based on the past information. 
We start with some basic definitions of important classes of policies: 
\begin{definition}
For an MDP, a \textbf{randomized stationary policy} $\pi$ defines an algorithm which, whenever the system is in state $s \in \mathcal{S}$, chooses an action $a \in \mathcal{A}$ according to a fixed conditional probability function $\pi(a|s)$, defined for all $a\in\mathcal{A}$ and $s\in\mathcal{S}$. 
\end{definition}

\begin{definition}\label{def:pp}
For an MDP, a \textbf{pure policy} $\pi$ is a randomized stationary policy with all probabilities equal to either 0 or 1.  That is, 
a pure policy is defined by a  deterministic mapping between states $s \in \mathcal{S}$ and actions $a \in \mathcal{A}$. 
Whenever
the system is in a state $s \in \mathcal{S}$, it always chooses a particular action $a_s \in\mathcal{A}$ (with probability 1).
\end{definition}

Note that if an MDP has a finite state and action space, the set of all pure policies is also finite. 
Consider the MDP associated with a particular system $k \in \{1, \ldots, K\}$. For any randomized stationary policy $\pi$, it holds that $\sum_{a\in\mathcal{A}^{(k)}}\pi(a|s) = 1$ for all  $s\in\mathcal{S}^{(k)}$. Define the transition probability matrix $\mathbf{P}_{\pi}^{(k)}$ under policy $\pi$ to have components as follows:
\begin{equation}\label{transition-matrix}
P_{\pi}^{(k)}(s,s') = \sum_{a\in\mathcal{A}^{(k)}}\pi(a|s)P_{a}^{(k)}(s,s'),~~s,s'\in\mathcal{S}^{(k)}.
\end{equation}
It is easy to verify that $\mathbf{P}_{\pi}^{(k)}$ is indeed a \emph{stochastic matrix}, that is, it 
has rows with nonnegative components that 
sum to 1. Let  $d_0^{(k)}\in[0,1]^{|\mathcal{S}^{(k)}|}$ be an (arbitrary) initial distribution for the $k$-th MDP.\footnote{For any set $\mathcal{S}$, 
we use $|\mathcal{S}|$ to denote the cardinality of the set.} 
Define the state distribution 
at time $t$ under $\pi$ as $d_{\pi,t}^{(k)}$. By the Markov property of the system, 
we have $d_{\pi,t}^{(k)} = d_{0}^{(k)}\l(\mathbf{P}_{\pi}^{(k)}\r)^t$.    A transition probability matrix $\mathbf{P}_{\pi}^{(k)}$ is 
\emph{ergodic} if it gives rise to a Markov chain that is irreducible and aperiodic.  Since the state space is finite, an ergodic 
matrix $\mathbf{P}_{\pi}^{(k)}$ has a unique stationary distribution denoted $d_{\pi}^{(k)}$, so that  
$d_{\pi}^{(k)}$ is the unique probability vector solving $d=d \mathbf{P}_{\pi}^{(k)}$.

\begin{assumption}[Unichain model]\label{assumption-1}
There exists a universal integer $\widehat{r}\geq1$ such that for any integer $r\geq \widehat{r}$ and every $k \in \{1, \ldots, K\}$, 
 we have the product $\mathbf{P}_{\pi_1}^{(k)}\mathbf{P}_{\pi_2}^{(k)}\cdots \mathbf{P}_{\pi_r}^{(k)}$ is \xcolor{a transition matrix with strictly positive entries} for any sequence of pure policies 
$\pi_1,\pi_2,\cdots,\pi_r$ associated with the $k$th MDP. 
\end{assumption}

\begin{remark}
Assumption \ref{assumption-1}  implies that each MDP $k\in\{1, \ldots, K\}$  is ergodic under any pure policy.  This follows
by taking $\pi_1,\pi_2,\cdots,\pi_r$ all the same in Assumption \ref{assumption-1}. \xcolor{Since the transition matrix of any randomized stationary policy can be formed as a convex combination of those of pure policies, any randomized stationary policy results in an ergodic MDP  for which there is a unique stationary distribution.} Assumption \ref{assumption-1} is easy to check via the following simple sufficient condition.

\end{remark}

\begin{proposition}
Assumption \ref{assumption-1} holds if, for every $k \in \{1, \ldots, K\}$,  
there is a fixed ergodic matrix $\mathbf{P}^{(k)}$ (i.e., a transition probability matrix that defines an irreducible and aperiodic
Markov chain) such that for any pure policy $\pi$ on MDP $k$ we have the decomposition 
$$\mathbf{P}_{\pi}^{(k)} = \delta_{\pi}\mathbf{P}^{(k)} + (1-\delta_{\pi})\mathbf{Q}_{\pi}^{(k)},$$
where $\delta_{\pi}\in(0,1]$ depends on the pure policy $\pi$ and $\mathbf{Q}_{\pi}^{(k)}$ is a stochastic matrix depending on $\pi$.
\end{proposition}
\begin{proof}
Fix $k \in \{1, \ldots, K\}$ and assume every pure policy on MDP $k$ has the above decomposition. 
Since there are only finitely many pure policies, there exists a lower bound $\delta_{\min}>0$ such that $\delta_{\pi}\geq\delta_{\min}$ for every pure policy $\pi$. 
Since $\mathbf{P}^{(k)}$ is an ergodic matrix, there exists an integer $r^{(k)}>0$ large enough such that $(\mathbf{P}^{(k)})^r$ has strictly positive components for all $r \geq r^{(k)}$.  Fix $r \geq r^{(k)}$ and 
let $\pi_1, \ldots, \pi_r$ be any sequence of $r$ pure policies on MDP $k$. Then
\[\mathbf{P}_{\pi_1}^{(k)}\cdots\mathbf{P}_{\pi_r}^{(k)}\geq \delta_{\min}\l(\mathbf{P}^{(k)}\r)^r > 0\]
The universal integer $\hat{r}$ can be taken as the maximum integer $r^{(k)}$ over all $k \in \{1, \ldots, K\}$. 
\end{proof}

\begin{definition}\label{def:jrsp}
A \textbf{joint randomized stationary policy} $\Pi$ on $K$ parallel MDPs defines an algorithm which chooses a joint action $\mathbf{a}:=\l(a^{(1)},~a^{(2)},~\cdots,~a^{(K)}\r)
\in\mathcal{A}^{(1)}\times\mathcal{A}^{(2)}\cdots\times\mathcal{A}^{(K)}$ given the joint state $\mathbf{s}:=\l(s^{(1)},~s^{(2)},~,\cdots,s^{(K)}\r)
\in\mathcal{S}^{(1)}\times\mathcal{S}^{(2)}\cdots\times\mathcal{S}^{(K)}$ according to a fixed conditional probability 
$\Pi\l(\mathbf{a} \l| \mathbf{s} \r.\r)$.
\end{definition}

The following special class of \emph{separable}  policies can be implemented separately over each of the $K$ MDPs and plays a role in both algorithm design and performance analysis.

\begin{definition}\label{def:rsp}
A joint randomized stationary policy $\pi$ is \textbf{separable} if the conditional probabilities 
 $\pi:=\l(\pi^{(1)},~\pi^{(2)},~\cdots,~\pi^{(K)}\r)$ decompose as a product 
 $$\pi\l(\mathbf{a} \l| \mathbf{s} \r.\r) = \prod_{k=1}^K \pi^{(k)}\l(a^{(k)}|s^{(k)}\r)$$
 for all $\mathbf{a} \in \mathcal{A}^{(1)}\times \cdots\times\mathcal{A}^{(K)}$, $\mathbf{s} \in \mathcal{S}^{(1)}\cdots\times\mathcal{S}^{(K)}$.
 \end{definition}

\subsection{Technical assumptions}

The functions $f_{t}^{(k)}$ and $g_{i,t}^{(k)}$ are determined by random processes defined over $t= 0,1,2,\cdots$. Specifically, 
let $\Omega$ be a finite dimensional vector space. Let $\{\omega_t\}_{t=0}^{\infty}$ and $\{\mu_t\}_{t=0}^{\infty}$ be two sequences of random vectors in $\Omega$. Then for all $a\in\mathcal{A}^{(k)}$, $s\in\mathcal{S}^{(k)}$, $i\in\{1,2,\cdots,m\}$ we have
 \begin{align*}
 &g_{i,t}^{(k)}(a,s) = \hat{g}_i^{(k)}\l(a,s,\omega_t\r),\\
 &f_{t}^{(k)}(a,s) =  \hat{f}^{(k)}\l(a,s,\mu_t\r)
 \end{align*}
 where $\hat{g}_i^{(k)}$ and $\hat{f}^{(k)}$ formally define the time-varying functions in terms of the random processes $\omega_t$ and $\mu_t$. 
 It is assumed that the processes $\{\omega_t\}_{t=0}^{\infty}$ and $\{\mu_t\}_{t=0}^{\infty}$ are generated at the start of slot $0$ (before any control actions are taken),
 and 
revealed gradually over time, 
 so that functions $g_{i,t}^{(k)}$ and $f_t^{(k)}$ are only revealed at the end of slot $t$.
 \begin{remark}
 The functions generated in this way are also called \textit{oblivious functions}. Such an assumption is commonly adopted in previous unconstrained online MDP works (e.g. \cite{even2009online}, \cite{yu2009markov} and \cite{dick2014online}). Further, it is also shown in \cite{yu2009markov} that without this assumption, one can choose a sequence of objective functions against the decision maker in a specifically designed MDP scenario so that one never achieves the sublinear regret. 
 \end{remark}

 The functions are also assumed to be bounded by a universal constant $\Psi$, so that: 
 \begin{equation}\label{function-bounds}
  |\hat{g}_i^{(k)}(a,s,\omega)| \leq \Psi, |\hat{f}^{(k)}(a,s,\mu)| \leq \Psi \quad, \forall k \in \{1, \ldots, K\},  \forall a \in \mathcal{A}^{(k)}  , s \in \mathcal{S}^{(k)},  \forall \omega, \mu \in \Omega.  
 \end{equation}
 It is assumed that $\{\omega_t\}_{t=0}^{\infty}$ is independent, identically distributed (i.i.d.) and independent of $\{\mu_t\}_{t=0}^{\infty}$. Hence, 
 the constraint functions can be arbitrarily correlated on the same slot, but appear i.i.d. over different slots. 
On the other hand, no specific model is imposed on $\{\mu_t\}_{t=0}^{\infty}$. Thus,  
the functions $f_{t}^{(k)}$ can be arbitrarily time varying. 
Let 
$\mathcal{H}_t$ be the system information up to time $t$, then, for any $t\in\{0,1,2,\cdots\}$, $\mathcal{H}_t$ contains
state and action information up to time $t$, i.e.
$\mathbf{s}_0,\cdots,\mathbf{s}_t$, $\mathbf{a}_0,\cdots,\mathbf{a}_t$, and $\{\omega_t\}_{t=0}^{\infty}$ and $\{\mu_t\}_{t=0}^{\infty}$.
Throughout this paper, we make the following assumptions.

\begin{assumption}[Independent transition]\label{assumption:indep-trans}
For each MDP, given the state $s^{(k)}_t\in\mathcal{S}^{(k)}$ and action $a^{(k)}_t\in\mathcal{A}^{(k)}$, the next state $s^{(k)}_{t+1}$ is independent of all other past information up to time $t$ as well as the state transition $s_{t+1}^{(j)},~\forall j\neq k$, i.e., for all
$s \in \mathcal{S}^{(k)}$ it holds that 
$$Pr\l( s_{t+1}^{(k)}=s | \mathcal{H}_t, s_{t+1}^{(j)},~\forall j\neq k \r)=Pr\l( s_{t+1}^{(k)}=s | s_t^{(k)},a_t^{(k)}\r)$$
where 
$\mathcal{H}_t$ contains all past information up to time $t$.
\end{assumption}

Intuitively, this assumption means that all MDPs are running independently in the joint probability space and thus the only coupling among them comes from the constraints, which reflects the notion of \textit{weakly coupled MDPs} in our title. Furthermore, by definition of $\mathcal{H}_t$, given $s_t^{(k)},a_t^{(k)}$, the next transition $s_{t+1}^{(k)}$ is also independent of function paths $\{\omega_t\}_{t=0}^{\infty}$ and $\{\mu_t\}_{t=0}^{\infty}$. 

The following assumption states the constraint set is strictly feasible.

\begin{assumption}[Slater's condition]\label{assumption:slater}
There exists a real value $\eta>0$ and a fixed separable randomized stationary policy $\widetilde{\pi}$ such that
\[
\mathbb{E}\left[\sum_{k=1}^Kg_{i,t}^{(k)}\l(a^{(k)}_t,s^{(k)}_t\r)\Big|~ d_{\widetilde{\pi}},\widetilde{\pi}\right]\leq-\eta,~\forall i\in\{1,2,\cdots,m\},
\]
\xcolor{
where the initial state is $d_{\widetilde{\pi}}$ and is the unique stationary distribution of policy $\widetilde{\pi}$, and the expectation is taken with respect to the random initial state and the stochastic function $g_{i,t}^{(k)}(a,s)$ (i.e., $w_t$).  }
\end{assumption}

Slater's condition is a common assumption in convergence time analysis of constrained convex optimization (e.g. \cite{nedic2009approximate}, \cite{bertsekas2009convex}). 
Note that this assumption readily implies the constraint set $\mathcal{G}$ can be achieved by the above randomized stationary policy. Specifically, take $d_0^{(k)}=d_{\widetilde\pi^{(k)}}$ and $\mathscr{P}=\widetilde{\pi}$, then, we have
\[G_{i,T}(d_0,{\tilde\pi}) = \sum_{t=0}^{T-1}\mathbb{E}\left[\sum_{k=1}^Kg_{i,t}^{(k)}\l(a^{(k)}_t,s^{(k)}_t\r)\Big|~ d_{\widetilde{\pi}},\widetilde{\pi}\right]\leq -\eta T<0.
\]

\subsection{The state-action polyhedron}\label{sec:sapoly}
In this section, we recall the well-known linear program formulation of an MDP (see, for example, \cite{altman1999constrained} and \cite{fox1966markov}).
Consider an MDP with a state space $\mathcal{S}$ and an action space $\mathcal{A}$.
Let $\Delta\subseteq \mathbb{R}^{|\mathcal S| |\mathcal{A}|}$ be a probability simplex, i.e.
\[\Delta=\l\{\theta\in\mathbb{R}^{|\mathcal S| |\mathcal{A}|}:~\sum_{(s,a)\in\mathcal{S}\times\mathcal{A}}\theta(s,a)=1,~\theta(s,a)\geq0\r\}.\]
Given a randomized stationary policy $\pi$ with stationary state distribution $d_\pi$, the MDP is a Markov chain with transition matrix $\mathbf{P}_\pi$ given by \eqref{transition-matrix}. Thus, it must satisfy the following balance equation:
\[
\sum_{s\in\mathcal{S}}d_\pi(s)P_\pi(s,s') =  d_\pi(s'),~\forall s'\in\mathcal{S}.
\]
Defining $\theta(a,s) = \pi(a|s)d_\pi(s)$ and substituting the definition of transition probability \eqref{transition-matrix} into the above equation gives
\[
\sum_{s\in\mathcal{S}}\sum_{a\in\mathcal{A}}\theta(s,a)P_a(s,s') = \sum_{a\in\mathcal{A}}\theta(s',a),
~~\forall s'\in\mathcal{S}.
\]
The variable $\theta(a,s)$ is often interpreted as a stationary probability of being at state $s\in\mathcal{S}$ and taking action $a\in\mathcal{A}$ under some randomized stationary policy. 
The state action polyhedron $\Theta$ is then defined as 
\[
\Theta := \l\{\theta\in\Delta:~\sum_{s\in\mathcal{S}}\sum_{a\in\mathcal{A}}\theta(s,a)P_a(s,s') = \sum_{a\in\mathcal{A}}\theta(s',a),
~~\forall s'\in\mathcal{S}  \r\}.
\]
Given any $\theta\in\Theta$, one can recover a randomized stationary policy $\pi$ at any state $s\in\mathcal{S}$ as 
\begin{equation}\label{solution-to-LP}
\pi(a|s) =
\begin{cases}
 \frac{\theta(a,s)}{\sum_{a\in\mathcal{A}}\theta(a,s)},~~&\textrm{if}~\sum_{a\in\mathcal{A}}\theta(a,s)\neq0,\\
 0,~~&\textrm{otherwise}.
\end{cases}
\end{equation}

Given any fixed penalty function $f(a,s)$, the best policy minimizing the penalty (without constraint) is a randomized stationary policy given by the solution to the following linear program (LP):
\begin{align}\label{LP1}
\min~~\langle\mathbf{f},\theta\rangle,~~s.t.~~\theta\in\Theta.
\end{align}
where $\mathbf{f}:=[f(a,s)]_{a\in\mathcal{A},~s\in\mathcal{S}}$. Note that for any policy $\pi$ given by the state-action pair $\theta$ according to \eqref{solution-to-LP},
\begin{align*}
\l\langle \mathbf{f},\theta \r\rangle = \mathbb{E}_{s\sim d_\pi, a\sim \pi(\cdot | s)}\l[f(a,s)\r],
\end{align*}
Thus, $\l\langle \mathbf{f},\theta \r\rangle$ is often referred to as the stationary state penalty of policy $\pi$.

It can also be shown that any state-action pair in the set $\Theta$ can be achieved by a convex combination of state-action vectors of pure policies, and thus all corner points of the polyhedron $\Theta$ are from pure policies. As a consequence, the best randomized stationary policy solving \eqref{LP1} is always a pure policy. 

\subsection{Preliminary results on MDPs}\label{sec:prelim-thm}
In this section, we give preliminary results regarding the properties of our weakly coupled MDPs under randomized stationary policies. The proofs can be found in the appendix. We start with a lemma on the uniform mixing of MDPs.

\begin{lemma}\label{lemma:mixing1}
\xcolor{Suppose Assumption \ref{assumption-1} and \ref{assumption:indep-trans} hold. }
There exists a positive integer $r$ and a constant $\tau\geq1$ such that for any two state distributions $d_1$ and $d_2$,
\[
\sup_{\pi_1^{(k)},\cdots,\pi_r^{(k)}}\l\|  \l(d_1^{(k)}-d_2^{(k)}\r)\mathbf{P}_{\pi_1^{(k)}}^{(k)}\mathbf{P}_{\pi_2^{(k)}}^{(k)}\cdots \mathbf{P}_{\pi_r^{(k)}}^{(k)}  \r\|_1\leq 
e^{-1/\tau}\l\| d_1^{(k)}-d_2^{(k)}\r\|_1,~\forall k\in\{1,2,\cdots,K\}
\]
where the supremum is taken with respect to \textbf{any} sequence of $r$ randomized stationary policies $\l\{\pi_1^{(k)},\cdots,\pi_r^{(k)}\r\}$.
\end{lemma}

\begin{lemma}\label{lemma:prod-chain}
\xcolor{Suppose Assumption \ref{assumption-1} and \ref{assumption:indep-trans} hold.}
Consider the product MDP with product state space $\mathcal{S}^{(1)}\times\cdots\times \mathcal{S}^{(K)}$ and action space $\mathcal{A}^{(1)}\times\cdots\times \mathcal{A}^{(K)}$. Then, the following hold:
\begin{enumerate}
\item  The product MDP is irreducible and aperiodic under any joint randomized stationary policy.
\item The stationary state-action probability $\{\theta^{(k)}\}_{k=1}^K$ of any joint randomized stationary policy satisfies $\theta^{(k)}\in\Theta^{(k)},~\forall k\in\{1,2,\cdots,K\}$.
\end{enumerate}
\end{lemma}

An immediate conclusion we can draw from this lemma is that given any penalty and constraint functions $\mathbf{f}^{(k)}$ and $\mathbf{g}_{i}^{(k)}$, $k=1,2,\cdots,K$, the stationary penalty and constraint value of any joint randomized stationary policy can be expressed as 
$$
\sum_{k=1}^K\l\langle\mathbf{f}^{(k)},\theta^{(k)} \r\rangle,~\sum_{k=1}^K\l\langle\mathbf{g}_i^{(k)},\theta^{(k)} \r\rangle,~~i=1,2,\cdots,m,
$$
with $\theta^{(k)}\in\Theta^{(k)}$. 
This in turn implies such stationary state-action probabilities $\{\theta^{(k)}\}_{k=1}^{K}$ can also be realized via a separable randomized stationary policy
$\pi$ with 
\begin{equation}\label{eq:relation}
\pi^{(k)}(a|s) = \frac{\theta^{(k)}(a,s)}{\sum_{a\in\mathcal{A}^{(k)}}\theta^{(k)}(a,s)},~a\in\mathcal{A}^{(k)},~s\in\mathcal{S}^{(k)},
\end{equation}
and the corresponding stationary penalty and constraint value can also be achieved via this policy. This fact suggests that when considering the stationary state performance only, the class of separable randomized stationary policies is large enough to cover all possible stationary penalty and constraint values.

\xcolor{
In particular, 
let $\tilde\pi=\l( \tilde\pi^{(1)},\cdots,\tilde\pi^{(K)}\r)$ be the separable randomized stationary policy associated with the Slater condition (Assumption \ref{assumption:slater}).  Using the fact that the constraint functions $\mathbf{g}_{i,t}^{(k)},k=1,2,\cdots,K$ (i.e. $w_t$) are i.i.d.and Assumption \ref{assumption:indep-trans} on independence of probability transitions, we have the constraint functions $g_{i,t}^{(k)}$ and the state-action pairs at any time $t$ are mutuallly independent.  Thus, }
\[
\mathbb{E}\left[\sum_{k=1}^Kg_{i,t}^{(k)}\l(a^{(k)}_t,s^{(k)}_t\r)\Big|~ d_{\widetilde{\pi}},\widetilde{\pi}\right]
= \sum_{k=1}^K\l\langle\expect{\mathbf{g}_{i,t}^{(k)}}, \tilde\theta^{(k)} \r\rangle,
\]
where $\tilde\theta^{(k)}$ corresponds to $\tilde\pi$ according to \eqref{eq:relation}.

Then, Slater's condition
can be translated to the following: There exists a sequence of state-action probabilities $\{\tilde\theta^{(k)}\}_{k=1}^{K}$ from a separable randomized stationary policy such that $\tilde\theta^{(k)}\in\Theta^{(k)},~\forall k$, and
\begin{equation}\label{slater-2}
\sum_{k=1}^K\l\langle\expect{\mathbf{g}_{i,t}^{(k)}}, \tilde\theta^{(k)} \r\rangle\leq-\eta,~~i=1,2,\cdots,m,
\end{equation}
The assumption on separability does not lose generality in the sense that if there is no separable randomized stationary policy that satisfies \eqref{slater-2}, then, there is no \textit{joint} randomized stationary policy that satisfies \eqref{slater-2} either.

\subsection{The blessing of slow-update property in online MDPs}\label{sec:dis-slow-change}
\xcolor{
 The current state of an MDP depends on previous states and actions. 
As a consequence, the slot $t$ penalty not only depends on the current penalty function and current action,  but also on the system history.  This complication does not arise in classical online convex optimization (\cite{hazan2016introduction},\cite{zinkevich2003online}) as there is no notion of ``state'' and the slot $t$ penalty depends only on the slot $t$ penalty function and action. 
}

\xcolor{
Now imagine a virtual system where, on each slot $t$, a policy $\pi_t$ is chosen (rather than an action).  Further imagine the MDP immediately reaching its corresponding stationary distribution $d_{\pi_t}$. Then the states and actions on previous slots do not matter and the slot $t$ performance depends only on the chosen policy $\pi_t$ and on the current penalty and constraint functions.  This imaginary system now has a structure similar to classical online convex optimization as in the Zinkevich scenario \cite{zinkevich2003online}. 
}

\xcolor{
A key feature of online convex optimization algorithms as in \cite{zinkevich2003online} is that they update their decision variables slowly.  For a fixed time scale $T$ over which 
$\mathcal{O}(\sqrt{T})$ regret is desired, the decision variables are typically changed no more than a distance $\mathcal{O}(1/\sqrt{T})$ from one slot to the next.  An important insight in prior (unconstrained) MDP works(e.g. \cite{even2009online}, \cite{yu2009markov} and \cite{dick2014online}) is that such slow updates also guarantee the ``approximate'' convergence of an MDP to its stationary distribution.  As a consequence, one can design the decision policies under the imaginary assumption that the system instantly reaches its stationary distribution, and later bound the error between the true system and the imaginary system.  If the error is on the same order as the desired $\mathcal{O}(\sqrt{T})$ regret, then this approach works.  This idea serves as a cornerstone of our algorithm design of the next section, which treats the case of multiple weakly coupled systems with both objective functions and constraint functions. 
}

\section{OCMDP algorithm}\label{sec:algorithm}
\xcolor{
Our proposed algorithm is distributed in the sense that each time slot, each MDP solves its own subproblem and the constraint violations are controlled by a simple update of global multipliers called ``virtual queues'' at the end of each slot.}
Let $\Theta^{(1)},~\Theta^{(2)},~\cdots,~\Theta^{(K)}$ be state-action polyhedron of $K$ MDPs, respectively.
Let $\theta_t^{(k)}\in\Theta^{(k)}$ be a state-action vector at time slot $t$. 
At $t = 0$, each MDP chooses its initial state-action vector $\theta_0^{(k)}$ resulting from any \textit{separable} randomized stationary policy $\pi_0^{(k)}$.
For example, one could choose a uniform policy $\pi^{(k)}(a|s) = 1/\l|\mathcal{A}^{(k)}\r|,~\forall s\in\mathcal{S}^{(k)}$, solve the equation $d_{\pi_0^{(k)}} = d_{\pi_0^{(k)}}\mathbf{P}_{\pi_0^{(k)}}^{(k)}$ to get a probability vector $d_{\pi_0^{(k)}}$, and obtain $\theta_0^{(k)}(a,s) = d_{\pi_0^{(k)}}(s)/\l|\mathcal{A}^{(k)}\r|$. For each constraint $i\in\{1,2,\cdots,m\}$, let $Q_i(t)$ be a \textit{virtual queue} defined over slots $t=0,1,2,\cdots$ with the initial condition $Q_i(0) = Q_i(1) = 0$, and update equation:
\begin{equation}\label{Q-update}
Q_i(t+1) = \max\left\{ Q_i(t) + \sum_{k=1}^K \l\langle\mathbf{g}_{i,t-1}^{(k)},\theta_t\r\rangle,~0\right\},~\forall t\in\{1,2,3,\cdots\}.
\end{equation}
Our algorithm uses two parameters $V>0$ and $\alpha>0$ and makes decisions as follows:
\begin{itemize}
\item At the start of each slot $t\in\{1,2,3,\cdots\}$, the $k$-th MDP observes $Q_i(t),~i=1,2,\cdots,m$ and chooses $\theta_t^{(k)}$ to solve the following subproblem:
\begin{equation}\label{optimize}
\theta_t^{(k)} = \textrm{argmin}_{\theta\in\Theta^{(k)}}\left\langle 
V\mathbf{f}_{t-1}^{(k)}+\sum_{i=1}^mQ_i(t)\mathbf{g}_{i,t-1}^{(k)},\theta\right\rangle
+\alpha\l\|\theta-\theta_{t-1}^{(k)}\r\|_2^2.
\end{equation}
\item Construct the randomized stationary policy $\pi_t^{(k)}$ according to \eqref{solution-to-LP} with 
$\theta=\theta_t^{(k)}$, and choose the action $a_t^{(k)}$ at $k$-th MDP according to the conditional distribution $\pi_t^{(k)}\l(\cdot|s^{(k)}_t\r)$. 
\item Update the virtual queue $Q_i(t)$ according to \eqref{Q-update} for all $i=1,2,\cdots,m$.
\end{itemize}

Note that for any slot $t\geq1$, this algorithm gives a \textit{separable} randomized stationary policy, so that each MDP chooses its own policy based on its own function 
$\mathbf{f}_{t-1}^{(k)}$, $\mathbf{g}_{i,t-1}^{(k)},i\in\{1,2,\cdots,m\}$, and a common multiplier 
$\mathbf{Q}(t) := \l(Q_1(t),\cdots,Q_m(t)\r)$.  

The next lemma shows that solving \eqref{optimize} is in fact a projection onto the state-action polyhedron. For any set $\mathcal{X}\in\mathbb{R}^n$ and a vector $\mathbf{y}\in\mathbb{R}^n$, define the projection operator $\mathcal{P}_\mathcal{X}(\mathbf{y})$ as 
\[
\mathcal{P}_\mathcal{X}(\mathbf{y}) = \textrm{arginf}_{\mathbf{x}\in\mathcal{X}}\|\mathbf{x} - \mathbf{y}\|_2.
\]
\begin{lemma}
Fix an $\alpha>0$ and $t\in\{1,2,3,\cdots\}$. The $\theta_t$ that solves \eqref{optimize} is
\[
\theta_t^{(k)} = \mathcal{P}_{\Theta^{(k)}}\left( \theta_{t-1}^{(k)} - \frac{\mathbf{w}_t^{(k)}}{2\alpha}\right),
\]
where 
$
\mathbf{w}_t^{(k)} = V\mathbf{f}_{t-1}^{(k)}+\sum_{i=1}^mQ_i(t)\mathbf{g}_{i,t-1}^{(k)}\in\mathbb{R}^{|\mathcal{A}^{(k)}||\mathcal{S}^{(k)}|}.
$
\end{lemma}
\begin{proof}
By definition, we have
\begin{align*}
\theta_t^{(k)} =& \textrm{argmin}_{\theta\in\Theta^{(k)}} \l\langle \mathbf{w}_t^{(k)} , \theta \r\rangle + \alpha \l\| \theta - \theta^{(k)}_{t-1} \r\|_2^2\\
=& \textrm{argmin}_{\theta\in\Theta^{(k)}}   \l\langle \mathbf{w}_t^{(k)} , \theta - \theta^{(k)}_{t-1} \r\rangle + \alpha \l\| \theta - \theta^{(k)}_{t-1} \r\|_2^2
+   \l\langle \mathbf{w}_t^{(k)}, \theta^{(k)}_{t-1} \r\rangle\\
=& \textrm{argmin}_{\theta\in\Theta^{(k)}}~  \alpha\cdot \l(\l\langle \frac{\mathbf{w}_t^{(k)}}{\alpha} , \theta - \theta^{(k)}_{t-1} \r\rangle + \l\| \theta - \theta^{(k)}_{t-1} \r\|_2^2\r)
+   \l\langle \mathbf{w}_t^{(k)}, \theta^{(k)}_{t-1} \r\rangle\\
=& \textrm{argmin}_{\theta\in\Theta^{(k)}}~ \alpha\cdot\l\| \theta -  \theta^{(k)}_{t-1} + \frac{\mathbf{w}_t^{(k)}}{2\alpha} \r\|_2^2
=\mathcal{P}_{\Theta^{(k)}}\left( \theta_{t-1}^{(k)} - \frac{\mathbf{w}_t^{(k)}}{2\alpha}\right),
\end{align*}
finishing the proof.
\end{proof}

\subsection{Intuition of the algorithm and roadmap of analysis}
The intuition of this algorithm follows from the discussion in Section \ref{sec:dis-slow-change}. Instead of the Markovian regret \eqref{main-regret} and constraint set \eqref{main-constraint}, we work on the imaginary system that after the decision maker chooses any joint policy $\Pi_t$ and the penalty/constraint functions are revealed, the $K$ parallel Markov chains reach stationary state distribution right away, with state-action probability vectors $\l\{\theta_t^{(k)}\r\}_{k=1}^K$ for $K$ parallel MDPs. Thus there is no Markov state in such a system anymore and the corresponding stationary
 penalty and constraint function value at time $t$ can be expressed as $\sum_{k=1}^K\l\langle\mathbf{f}_t^{(k)},\theta_t^{(k)}\r\rangle$ and 
$\sum_{k=1}^K\l\langle\mathbf{g}_{i,t}^{(k)},\theta_t^{(k)}\r\rangle,~i=1,2,\cdots,m$, respectively. As a consequence, we are now facing a relatively easier task 
of minimizing the following regret:
\begin{equation}\label{stat-regret}
\sum_{t=0}^{T-1}\sum_{k=1}^K\expect{\l\langle\mathbf{f}_t^{(k)},\theta_t^{(k)}\r\rangle} - \sum_{t=0}^{T-1}\sum_{k=1}^K\expect{\l\langle\mathbf{f}_t^{(k)},\theta^{(k)}_*\r\rangle},
\end{equation}
where $\l\{\theta^{(k)}_*\r\}_{k=1}^K$ are the state-action probabilities corresponding to the best fixed joint randomized stationary policy within the following stationary constraint set \begin{equation}\label{stat-constraint}
\overline{\mathcal{G}}:=\left\{ \theta^{(k)}\in\Theta^{(k)},~k\in\{1,2,\cdots,K\}:~\sum_{k=1}^K\l\langle\expect{\mathbf{g}_{i,t}^{(k)}},\theta^{(k)}\r\rangle \leq 0,~i=1,2,\cdots,m\right\},
\end{equation}
with the assumption that Slater's condition \eqref{slater-2} holds.

\xcolor{
To analyze the proposed algorithm, we need to tackle the following two major challenges:
\begin{itemize}
\item Whether or not the policy decision of the proposed algorithm would yield $\mathcal{O}(\sqrt{T})$ regret and constraint violation on the imaginary system that reaches steady state instantaneously on each slot.
\item Whether the error between the imaginary and true systems can be bounded by $\mathcal{O}(\sqrt{T})$. 
\end{itemize}
}

\xcolor{
In the next section, we answer these questions via a multi-stage analysis piecing together the results of MDPs from Section \ref{sec:prelim-thm} with
multiple ingredients from convex analysis and stochastic queue analysis. We first show the 
$\mathcal{O}(\sqrt{T})$ regret and constraint violation in the imaginary online linear program incorporating a new regret analysis procedure with a stochastic drift analysis for queue processes. 
Then, we show if the benchmark randomized stationary algorithm always starts from its stationary state, then,
the discrepancy of regrets between the imaginary and true systems can be controlled via the slow-update property of the proposed algorithm together with the properties of MDPs developed in Section \ref{sec:prelim-thm}. Finally, for the problem with arbitrary non-stationary starting state, we reformulate it as a perturbation on the aforementioned stationary state problem and 
analyze the perturbation via Farkas' Lemma.
}

\section{Convergence time analysis}\label{sec:convergence-analysis}

\subsection{Stationary state performance: An online linear program}\label{sec:stat-analysis}


Let $\mathbf{Q}(t):=[Q_1(t),~Q_2(t),~\cdots,~Q_m(t)]$ be the virtual queue vector and $L(t) = \frac12\|\mathbf{Q}(t)\|_2^2$. Define the drift
$\Delta(t) := L(t+1) - L(t)$.

\subsubsection{Sample-path analysis}
This section develops a couple of bounds given a sequence of penalty functions $f_{0}^{(k)},f_{1}^{(k)},\cdots,f_{T-1}^{(k)}$ and constraint functions $g_{i,0}^{(k)},g_{i,1}^{(k)},\cdots,g_{i,T-1}^{(k)}$.
The following lemma provides bounds for virtual queue processes:
\begin{lemma}\label{lemma:pre-Q-bound}
For any $i\in\{1,2,\cdots,m\}$ at $T\in\{1,2,\cdots\}$, the following holds under the virtual queue update \eqref{Q-update},
\[
\sum_{t=1}^{T}\sum_{k=1}^K\l\langle \mathbf{g}_{i,t-1}^{(k)}, \theta_{t-1}^{(k)} \r\rangle 
\leq Q_i(T+1) - Q_i(1) +  \Psi\sum_{t=1}^T\sum_{k=1}^K\sqrt{\l| \mathcal{A}^{(k)} \r|\l| \mathcal{S}^{(k)} \r|}\l\| \theta^{(k)}_t - \theta^{(k)}_{t-1} \r\|_2,
\]
\xcolor{where $\Psi>0$ is the constant defined in \eqref{function-bounds}.}
\end{lemma}
\begin{proof}
By the queue updating rule \eqref{Q-update}, for any $t\in\mathbb{N}$,
\begin{align*}
Q_i(t+1) =& \max\l\{ Q_i(t) + \sum_{k=1}^K\l\langle \mathbf{g}_{i,t-1}^{(k)}, \theta_{t}^{(k)} \r\rangle,0 \r\} \\
\geq& Q_i(t) + \sum_{k=1}^K\l\langle \mathbf{g}_{i,t-1}^{(k)}, \theta_{t}^{(k)} \r\rangle\\
=& Q_i(t) + \sum_{k=1}^K\l\langle \mathbf{g}_{i,t-1}^{(k)}, \theta_{t-1}^{(k)} \r\rangle + \sum_{k=1}^K\l\langle \mathbf{g}_{i,t-1}^{(k)}, \theta_t^{(k)}-\theta_{t-1}^{(k)} \r\rangle\\
\geq&  Q_i(t) + \sum_{k=1}^K\l\langle \mathbf{g}_{i,t-1}^{(k)}, \theta_{t-1}^{(k)} \r\rangle - \sum_{k=1}^K\l\|g_{i,t-1}^{(k)}\r\|_2\l\|\theta_t^{(k)}-\theta_{t-1}^{(k)}\r\|_2, 
\end{align*}
Note that the constraint functions are deterministically bounded,
\[
\l\|g_{i,t-1}^{(k)}\r\|_2^2\leq \l| \mathcal{A}^{(k)} \r|\l| \mathcal{S}^{(k)} \r|\Psi^2.
\]
Substituting this bound into the above queue bound and rearranging the terms finish the proof.
\end{proof}

The next lemma provides a bound for the drift $\Delta(t)$.
\begin{lemma}\label{D-bound}
For any slot $t\geq1$, we have 
\[
\Delta(t)\leq \frac12mK^2\Psi^2+\sum_{i=1}^mQ_i(t)\sum_{k=1}^K\l\langle \mathbf{g}_{i,t-1}^{(k)}, \theta_{t}^{(k)} \r\rangle .
\]
\end{lemma}
\begin{proof}
By definition, we have
\begin{align*}
\Delta(t) =& \frac12\|\mathbf{Q}(t+1)\|_2^2 - \frac12\|\mathbf{Q}(t)\|_2^2\\
\leq&\frac12\sum_{i=1}^m\l( \l(Q_i(t) + \sum_{k=1}^K\l\langle \mathbf{g}_{i,t-1}^{(k)}, \theta_{t}^{(k)} \r\rangle \r)^2 - Q_i(t)^2 \r)\\
=& \sum_{i=1}^mQ_i(t)\sum_{k=1}^K\l\langle \mathbf{g}_{i,t-1}^{(k)}, \theta_{t}^{(k)} \r\rangle 
+ \frac12\sum_{i=1}^m\l(\sum_{k=1}^K\l\langle \mathbf{g}_{i,t-1}^{(k)}, \theta_{t}^{(k)} \r\rangle\r)^2.
\end{align*}
Note that by the queue update \eqref{Q-update}, we have
$$
\l| \sum_{k=1}^K \l\langle \mathbf{g}_{i,t-1}^{(k)}, \theta_{t}^{(k)} \r\rangle\r|
\leq K\l\| \mathbf{g}_{i,t-1}^{(k)} \r\|_\infty\l\| \theta_{t}^{(k)} \r\|_1\leq K\Psi.
$$
Substituting this bound into the drift bound finishes the proof.
\end{proof}

Consider a convex set $\mathcal{X}\subseteq\mathbb{R}^n$.
Recall that for a fixed real number $c>0$, a function $h:\mathcal{X}\rightarrow\mathbb{R}$ is said to be \textit{$c$-strongly convex}, if $h(x) - \frac{c}{2}\|x\|_2^2$ is convex over 
$x\in\mathcal{X}$. It is easy to see that if $q:\mathcal{X}\rightarrow\mathbb{R}$ is convex, $c>0$ and $b\in\mathbb{R}^n$, the function $q(x) + \frac{c}{2}\|x-b\|_2^2$ is $c$-strongly convex. Furthermore,  if the function $h$ is $c$-strongly convex that is minimized at a point $x_{\min}\in\mathcal{X}$, then (see, e.g., Corollary 1 in \cite{YuNeely17SIOPT}):
\begin{equation}\label{strongly-convex}
h(x_{\min})\leq h(y) - \frac{c}{2}\|y-x_{\min}\|_2^2, ~~\forall y\in\mathcal{X}.
\end{equation}
The following lemma is a direct consequence of the above strongly convex result. It also demonstrates the key property of our minimization subproblem \eqref{optimize}.

\begin{lemma}\label{strong-convex-bound}
The following bound holds for any $k\in\{1,2,\cdots,K\}$ and any fixed $\theta_*^{(k)}\in\Theta^{(k)}$:
\begin{multline}\label{DPP-bound}
V  \l\langle \mathbf{f}_{t-1}^{(k)}, \theta_t^{(k)} - \theta_{t-1}^{(k)} \r\rangle + \sum_{i=1}^mQ_i(t)
\l\langle\mathbf{g}_{i,t-1}^{(k)},\theta_t^{(k)}\r\rangle + \alpha\|\theta_t^{(k)}-\theta_{t-1}^{(k)}\|_2^2\\
\leq V\l\langle \mathbf{f}_{t-1}^{(k)}, \theta_*^{(k)} - \theta_{t-1}^{(k)} \r\rangle + \sum_{i=1}^mQ_i(t)
\l\langle\mathbf{g}_{i,t-1}^{(k)},\theta_*^{(k)}\r\rangle + \alpha\|\theta_*^{(k)} - \theta_{t-1}^{(k)}\|_2^2
-\alpha\|\theta_*^{(k)}-\theta_t^{(k)}\|_2^2.
\end{multline}
\end{lemma}
This lemma follows easily from the fact that the proposed algorithm \eqref{optimize} gives $\theta_t^{(k)}\in\Theta^{(k)}$ minimizing the left hand side, which is a strongly convex function, and then, applying \eqref{strongly-convex}, with 
\[
h\l(\theta^{(k)}_*\r) = V  \l\langle \mathbf{f}_{t-1}^{(k)}, \theta^{(k)}_* - \theta_{t-1}^{(k)} \r\rangle + \sum_{i=1}^mQ_i(t)
\l\langle\mathbf{g}_{i,t-1}^{(k)},\theta^{(k)}_*\r\rangle + \alpha\l\|\theta^{(k)}_*-\theta_{t-1}^{(k)}\r\|_2^2
\]

Combining the previous two lemmas gives the following ``drift-plus-penalty'' bound.
\begin{lemma}
For any fixed $\{\theta_*^{(k)}\}_{k=1}^K$ such that $\theta_*^{(k)}\in\Theta^{(k)}$ and $t\in\mathbb{N}$, we have the following bound,
\begin{multline}\label{dpp-bound}
\Delta(t) + V\sum_{k=1}^{K}\l\langle \mathbf{f}_{t-1}^{(k)}, \theta_t^{(k)} - \theta_{t-1}^{(k)} \r\rangle + \alpha\sum_{k=1}^K\|\theta^{(k)}_t-\theta^{(k)}_{t-1}\|_2^2\\
\leq \frac32mK^2\Psi^2 + V\sum_{k=1}^K\l\langle \mathbf{f}_{t-1}^{(k)}, \theta_*^{(k)} - \theta_{t-1}^{(k)} \r\rangle + \sum_{i=1}^mQ_i(t-1)\sum_{k=1}^K
\l\langle\mathbf{g}_{i,t-1}^{(k)},\theta_*^{(k)}\r\rangle \\
+ \alpha\sum_{k=1}^K\|\theta_*^{(k)} - \theta_{t-1}^{(k)}\|_2^2
-\alpha\sum_{k=1}^K\|\theta_*^{(k)}-\theta_t^{(k)}\|_2^2
\end{multline}
\end{lemma}
\begin{proof}
Using Lemma \ref{D-bound} and then Lemma \ref{strong-convex-bound}, we obtain
\begin{align}
&\Delta(t) + V\sum_{k=1}^{K}\l\langle \mathbf{f}_{t-1}^{(k)}, \theta_t^{(k)} - \theta_{t-1}^{(k)} \r\rangle + \alpha\sum_{k=1}^K\|\theta^{(k)}_t-\theta^{(k)}_{t-1}\|_2^2\nonumber\\
\leq& \frac12mK^2\Psi^2+\sum_{i=1}^mQ_i(t)\sum_{k=1}^K\l\langle \mathbf{g}_{i,t-1}^{(k)}, \theta_{t}^{(k)} \r\rangle
+ V\sum_{k=1}^{K}\l\langle \mathbf{f}_{t-1}^{(k)}, \theta_t^{(k)} - \theta_{t-1}^{(k)} \r\rangle + \alpha\sum_{k=1}^K\|\theta^{(k)}_t-\theta^{(k)}_{t-1}\|_2^2\nonumber\\
\leq& \frac12mK^2\Psi^2+ \sum_{k=1}^{K}\l\langle \mathbf{f}_{t-1}^{(k)}, \theta_*^{(k)} - \theta_{t-1}^{(k)} \r\rangle + \sum_{i=1}^mQ_i(t)
\sum_{k=1}^K\l\langle\mathbf{g}_{i,t-1}^{(k)},\theta_*^{(k)}\r\rangle \nonumber\\
&+ \alpha\sum_{k=1}^K\|\theta_*^{(k)} - \theta_{t-1}^{(k)}\|_2^2
-\alpha\sum_{k=1}^K\|\theta_*^{(k)}-\theta_t^{(k)}\|_2^2. \label{inter-2}
\end{align}
Note that by the queue updating rule \eqref{Q-update}, we have for any $t\geq2$,
\[
|Q_i(t)-Q_i(t-1)|\leq \l| \sum_{k=1}^K \l\langle \mathbf{g}_{i,t-2}^{(k)}, \theta_{t-1}^{(k)} \r\rangle\r|
\leq K\l\| \mathbf{g}_{i,t-2}^{(k)} \r\|_\infty\l\| \theta_{t-1}^{(k)} \r\|_1\leq K\Psi,
\]
and for $t=1$, $Q_i(t)-Q_i(t-1)=0$ by the initial condition of the algorithm. Also, we have for any $\theta_*^{(k)}\in\Theta^{(k)}$,
\[
\l|\sum_{k=1}^K\l\langle\mathbf{g}_{i,t-1}^{(k)},\theta_*^{(k)}\r\rangle   \r|  \leq K\l\| \mathbf{g}_{i,t-2}^{(k)} \r\|_\infty\l\| \theta_*^{(k)} \r\|_1\leq K\Psi.
\]
Thus, we have 
\begin{align*}
\sum_{i=1}^mQ_i(t)
\sum_{k=1}^K\l\langle\mathbf{g}_{i,t-1}^{(k)},\theta_*^{(k)}\r\rangle
\leq \sum_{i=1}^mQ_i(t-1)
\sum_{k=1}^K\l\langle\mathbf{g}_{i,t-1}^{(k)},\theta_*^{(k)}\r\rangle + mK^2\Psi^2.
\end{align*}
Substituting this bound into \eqref{inter-2} finishes the proof.
\end{proof}

\subsubsection{Objective bound}

\begin{theorem}\label{thm:stationary-regret}
For any $\{\theta_*^{(k)}\}_{k=1}^K$ in the constraint set \eqref{stat-constraint} and any $T\in\{1,2,3,\cdots\}$,
the proposed algorithm has the following stationary state performance bound:
\begin{multline*}
\frac1T\sum_{t=0}^{T-1}\expect{\sum_{k=1}^K\l\langle \mathbf{f}_t^{(k)},\theta_t^{(k)} \r\rangle} \leq \frac1T\sum_{t=0}^{T-1}\expect{\sum_{k=1}^K\l\langle \mathbf{f}_t^{(k)},\theta_*^{(k)} \r\rangle}\\
+\frac{2\alpha K}{TV}+\frac{mK^2\Psi^2}{T} + \frac{V\Psi^2}{2\alpha}\sum_{k=1}^K\l|\mathcal{S}^{(k)}\r|\l|\mathcal{A}^{(k)}\r|+ \frac32\frac{mK^2\Psi^2}{V},
\end{multline*}
In particular, choosing $\alpha=T$ and $V=\sqrt{T}$ gives the $\mathcal{O}(\sqrt{T})$ regret
\begin{multline*}
\frac1T\sum_{t=0}^{T-1}\expect{\sum_{k=1}^K\l\langle \mathbf{f}_t^{(k)},\theta_t^{(k)} \r\rangle} \leq \frac1T\sum_{t=0}^{T-1}\expect{\sum_{k=1}^K\l\langle \mathbf{f}_t^{(k)},\theta_*^{(k)} \r\rangle}\\
+\l(2K + \frac{\Psi^2}{2}\sum_{k=1}^K\l|\mathcal{S}^{(k)}\r|\l|\mathcal{A}^{(k)}\r|+ \frac52mK^2\Psi^2\r)\frac{1}{\sqrt{T}}.
\end{multline*}
\end{theorem}
\begin{proof}
\xcolor{First of all, note that $\{\mathbf{g}_{i,t-1}^{(k)}\}_{k=1}^K$ is i.i.d. and independent of all system history up to $t-1$, and thus independent of $Q_i(t-1),~i=1,2,\cdots,m$.} We have
\begin{equation}\label{neg-drift-2}
\expect{Q_i(t-1)\l\langle\mathbf{g}_{i,t-1}^{(k)},\theta^{(k)}_*\r\rangle} = \expect{Q_i(t-1)}\expect{\sum_{k=1}^K\l\langle\mathbf{g}_{i,t-1}^{(k)},\theta^{(k)}_*\r\rangle}\leq 0
\end{equation}
where the last inequality follows from the assumption that $\{\theta^{(k)}_*\}_{k=1}^K$ is in the constraint set \eqref{stat-constraint}. Substituting $\theta^{(k)}_*$ into \eqref{dpp-bound}
and taking expectation with respect to both sides give
\begin{multline*}
\expect{\Delta(t)} + V\expect{\sum_{k=1}^{K}\l\langle \mathbf{f}_{t-1}^{(k)}, \theta_t^{(k)} - \theta_{t-1}^{(k)} \r\rangle} + \alpha\expect{\sum_{k=1}^K\|\theta^{(k)}_t-\theta^{(k)}_{t-1}\|_2^2}\\
\leq \frac32mK^2\Psi^2 + V\expect{\sum_{k=1}^K\l\langle \mathbf{f}_{t-1}^{(k)}, \theta^{(k)}_* - \theta_{t-1}^{(k)} \r\rangle} 
+\sum_{i=1}^m\expect{Q_i(t-1)\sum_{k=1}^K\l\langle\mathbf{g}_{i,t-1}^{(k)},\theta^{(k)}_*\r\rangle}  
+ \alpha\expect{\sum_{k=1}^K\|\theta^{(k)}_* - \theta_{t-1}^{(k)}\|_2^2}\\
-\alpha\expect{\sum_{k=1}^K\|\theta^{(k)}_*-\theta_t^{(k)}\|_2^2}\\
\leq \frac32mK^2\Psi^2 + V\expect{\sum_{k=1}^K\l\langle \mathbf{f}_{t-1}^{(k)}, \theta^{(k)}_* - \theta_{t-1}^{(k)} \r\rangle} + \alpha\expect{\sum_{k=1}^K\|\theta^{(k)}_* - \theta_{t-1}^{(k)}\|_2^2}
-\alpha\expect{\sum_{k=1}^K\|\theta^{(k)}_*-\theta_t^{(k)}\|_2^2},
\end{multline*}
where the second inequality follows from \eqref{neg-drift-2}.
Note that for any $k$, completing the squares gives
\begin{align*}
V\l\langle \mathbf{f}_{t-1}^{(k)}, \theta_t^{(k)} - \theta_{t-1}^{(k)} \r\rangle+\alpha\|\theta^{(k)}_t-\theta^{(k)}_{t-1}\|_2^2
\geq \l\| \sqrt{\frac\alpha2}\l(\theta_t^{(k)}-\theta_{t-1}^{(k)}\r) + \frac{V}{2\sqrt{\alpha/2}}\mathbf{f}_{t-1}^{(k)} \r\|_2^2 - \frac{V^2\Psi^2\l|\mathcal{S}^{(k)}\r|\l|\mathcal{A}^{(k)}\r|}{2\alpha}.
\end{align*}
Substituting this inequality into the previous bound and rearranging the terms give
\begin{multline*}
V\expect{\sum_{k=1}^K\l\langle \mathbf{f}_{t-1}^{(k)}, \theta_{t-1}^{(k)} \r\rangle}\leq V\expect{\sum_{k=1}^K\l\langle \mathbf{f}_{t-1}^{(k)}, \theta_*^{(k)} \r\rangle}-\expect{\Delta(t)} 
+  \frac{V^2\sum_{k=1}^K\Psi^2\l|\mathcal{S}^{(k)}\r|\l|\mathcal{A}^{(k)}\r|}{2\alpha}+ \frac32mK^2\Psi^2 \\
+\alpha\expect{\sum_{k=1}^K\|\theta^{(k)}_* - \theta_{t-1}^{(k)}\|_2^2}
-\alpha\expect{\sum_{k=1}^K\|\theta^{(k)}_* - \theta_t^{(k)}\|_2^2}.
\end{multline*}
Taking telescoping sums from 1 to $T$ and dividing both sides by $TV$ gives,
\begin{align*}
\frac1T\sum_{t=1}^T\expect{\sum_{k=1}^K\l\langle \mathbf{f}_{t-1}^{(k)}, \theta_{t-1}^{(k)} \r\rangle}
\leq& \expect{\sum_{k=1}^K\l\langle \mathbf{f}_{t-1}^{(k)}, \theta^{(k)}_* \r\rangle} + \frac{L(0)-L(T+1)}{VT} + \frac{V\sum_{k=1}^K\Psi^2\l|\mathcal{S}^{(k)}\r|\l|\mathcal{A}^{(k)}\r|}{2\alpha} + \frac32\frac{mK^2\Psi^2}{V}\\
&+\frac{\alpha\expect{\sum_{k=1}^K\|\theta^{(k)}_* - \theta_{T-1}^{(k)}\|_2^2}
-\alpha\expect{\sum_{k=1}^K\|\theta^{(k)}_*-\theta_{T}^{(k)}\|_2^2}}{VT}\\
\leq&  \expect{\sum_{k=1}^K\l\langle \mathbf{f}_{t-1}^{(k)}, \theta^{(k)}_* \r\rangle} + \frac{V\sum_{k=1}^K\Psi^2\l|\mathcal{S}^{(k)}\r|\l|\mathcal{A}^{(k)}\r|}{2\alpha} + \frac32\frac{mK^2\Psi^2}{V}
+\frac{2\alpha K}{VT},
\end{align*}
where we use the fact that $L(0)=0$ and $ \|\theta_*^{(k)} - \theta_{T-1}^{(k)}\|_2^2\leq  \|\theta_*^{(k)} - \theta_{T-1}^{(k)}\|_1\leq2$.
\end{proof}

\subsubsection{A drift lemma and its implications}
From Lemma \ref{lemma:pre-Q-bound}, we know that in order to get the constraint violation bound, we need to look at the size of the virtual queue $Q_i(T+1),~i=1,2,\cdots,m$.
The following drift lemma serves as a cornerstone for our goal.
\begin{lemma}[Lemma 5 of \cite{hao2017onlinestochastic}] \label{lemma:drift-bound}
Let $\{Z(t), t\geq 1\}$ be a discrete time stochastic process adapted to a filtration $\{\mathcal{F}_{t-1}, t\geq 1\}$.  Suppose there exist integer $t_{0}>0$, real constants $\lambda\in \mathbb{R}$, $\delta_{\max} > 0$ and $0< \zeta \leq \delta_{\max}$ such that 
\begin{align} 
\vert Z(t+1) - Z(t) \vert \leq& \delta_{\max}, \\
\mathbb{E}[ Z(t+t_{0}) - Z(t) | \mathcal{F}_{t-1}] \leq & \left\{ \begin{array}{cc} t_{0}\delta_{\max}, &\text{if}~ Z(t) < \lambda \\  -t_{0}\zeta , &\text{if}~ Z(t) \geq \lambda \end{array}\right.. \label{eq:stochastic-process-drift-condition}
\end{align}
hold for all $t\in \{1,2,\ldots\}$. Then, the following holds
\begin{enumerate}
\item $\mathbb{E}[Z(t)] \leq  \lambda + t_{0}\frac{4\delta_{\max}^{2}}{\zeta} \log\big[1 + \frac{8\delta_{\max}^{2}}{\zeta^{2}}e^{\zeta/(4\delta_{\max})}\big], \forall t\in\{1,2,\ldots\}.$
\item For any constant $0<\mu<1$, we have $\text{Pr}(Z(t) \geq z) \leq \mu, \forall t\in\{1,2,\ldots\}$ where $z = \lambda + t_{0}\frac{4\delta_{\max}^{2}}{\zeta} \log\big[1 + \frac{8\delta_{\max}^{2}}{\zeta^{2}}e^{\zeta/(4\delta_{\max})}\big] +  t_{0}\frac{4\delta_{\max}^{2}}{\zeta} \log(\frac{1}{\mu})$.
\end{enumerate}
\end{lemma}

Note that a special case of above drift lemma for $t_0=1$ dates back to the seminal paper of Hajek (\cite{hajek1982hitting}) bounding the size of a random process with strongly negative drift. Since then, its power has been demonstrated in various scenarios ranging from steady state queue bound (\cite{eryilmaz2012asymptotically}) to feasibility analysis of stochastic optimization (\cite{wei2016online}). The current generalization to a multi-step drift is first considered in \cite{hao2017onlinestochastic}.

This lemma is useful in the current context due to the following lemma, whose proof can be found in the appendix.

\begin{lemma}\label{lemma:queue-drift-bound}
Let $\mathcal{F}_t,~t\geq1$ be the system history functions up to time $t$,  including $f^{(k)}_0,\cdots,f^{(k)}_{t-1}$, $g^{(k)}_{0,i},\cdots,g^{(k)}_{t-1,i}$, $i=1,2,\cdots,m,~k=1,2,\cdots,K$, and 
$\mathcal{F}_0$ is a null set. Let $t_0$ be an arbitrary positive integer, then, we have
\begin{align*}
\big \vert \Vert \mathbf{Q}(t+1)\Vert_2 - \Vert \mathbf{Q}(t)\Vert_2 \big\vert \leq& \sqrt{m}K\Psi, 
\end{align*}
\vskip -20pt
\begin{align*}
\mathbb{E}[\Vert \mathbf{Q}(t+t_0)\Vert_2 - \Vert \mathbf{Q}(t)\Vert_2 \big|  \mathcal{F}(t-1)] \leq & \left\{ \begin{array}{cc} t_0 \sqrt{m}K\Psi, &\text{if}~ \Vert \mathbf{Q}(t)\Vert <  \lambda \\  - t_0 \frac{\eta}{2}, &\text{if}~ \Vert \mathbf{Q}(t)\Vert \geq  \lambda
\end{array}\right.,
\end{align*}
where 
$\lambda =\frac{8VK\Psi  +  3mK^2\Psi^2 + 4K\alpha + t_0(t_0-1)m\Psi 
 + 2mK\Psi\eta t_0 + \eta^2 t_0^2}{\eta t_0}$.
\end{lemma}

Combining the previous two lemmas gives the virtual queue bound as 
\begin{multline*}
\expect{\|\mathbf{Q}(t)\|_2}\leq \frac{8VK\Psi  +  3mK^2\Psi^2 + 4K\alpha + t_0(t_0-1)m\Psi 
 + 2mK\Psi\eta t_0 + \eta^2 t_0^2}{\eta t_0}\\
 +4t_{0} \sqrt{m}K\Psi \log\big[1 + 8e^{1/4}\big].
\end{multline*}
We then choose $t_0=\sqrt{T}$, $V=\sqrt{T}$ and $\alpha=T$, which implies that
\begin{equation}\label{expected-Q-bound}
\expect{\|\mathbf{Q}(t)\|_2}\leq C(m,K,\Psi,\eta)\sqrt{T},
\end{equation}
where 
$$C(m,K,\Psi,\eta) = \frac{8K\Psi}{\eta} + \frac{3mK^2\Psi^2}{\eta^2} + \frac{4K+m\Psi}{\eta} + 2mK\Psi + \eta + 4\sqrt{m}K\Psi\log\l(1+8e^{1/4}\r). $$

\subsubsection{The slow-update condition and constraint violation}
In this section, we prove the slow-update property of the proposed algorithm, which not only implies the the $\mathcal{O}(\sqrt{T})$ constraint violation bound, but also plays a key role in Markov analysis.
\begin{lemma}\label{lemma:slow-update}
The sequence of state-action vectors $\theta_t^{(k)},~t\in\{1,2,\cdots,T\}$ satisfies 
\[
\expect{\|\theta_t^{(k)}-\theta_{t-1}^{(k)}\|_2}\leq \frac{\sqrt{m|\mathcal{A}^{(k)}||\mathcal{S}^{(k)}|}\Psi\expect{\|\mathbf{Q}(t)\|_2}}{2\alpha}
 + \frac{\sqrt{|\mathcal{A}^{(k)}||\mathcal{S}^{(k)}|}\Psi V}{2\alpha}.
\]
In particular,choosing $V=\sqrt{T}$ and $\alpha=T$ gives a slow-update condition
\begin{equation}\label{one-step-update}
\expect{\|\theta_t^{(k)}  -  \theta_{t-1}^{(k)}\|_2}\leq\frac{\sqrt{|\mathcal{A}^{(k)}||\mathcal{S}^{(k)}|}\Psi+C\sqrt{m|\mathcal{A}^{(k)}||\mathcal{S}^{(k)}|}\Psi}{2\sqrt{T}},
\end{equation}
where $C= C(m,K,\Psi,\eta)$ is defined in \eqref{expected-Q-bound}.
\end{lemma}
\begin{proof}
First, choosing $\theta = \theta_{t-1}$ in \eqref{DPP-bound} gives
\begin{multline*}
V\l\langle \mathbf{f}_{t-1}^{(k)}, \theta_t^{(k)}-\theta_{t-1}^{(k)} \r\rangle + \sum_{i=1}^mQ_i(t)
\l\langle\mathbf{g}_{i,t-1}^{(k)},\theta_t^{(k)}\r\rangle + \alpha\|\theta_t^{(k)}-\theta_{t-1}^{(k)}\|_2^2\\
\leq  \sum_{i=1}^mQ_i(t)
\langle\mathbf{g}_{i,t-1}^{(k)},\theta_{t-1}^{(k)}\rangle 
-\alpha\|\theta_{t-1}^{(k)}-\theta_t^{(k)}\|_2^2.
\end{multline*}
Rearranging the terms gives
\begin{align*}
2\alpha\|\theta_t^{(k)}-\theta_{t-1}^{(k)}\|_2^2\leq& - V\langle \mathbf{f}_{t-1}^{(k)}, \theta_t^{(k)}-\theta_{t-1}^{(k)}\rangle
-\sum_{i=1}^mQ_i(t)\langle\mathbf{g}_{i,t-1}^{(k)},\theta_t^{(k)} -\theta_{t-1}^{(k)}\rangle\\
\leq&V\|\mathbf{f}_{t-1}^{(k)}\|_2\cdot\|\theta_t^{(k)}-\theta_{t-1}^{(k)}\|_2 + \sum_{i=1}^mQ_i(t)
\|\mathbf{g}_{i,t-1}^{(k)}\|_2\cdot\|\theta_t^{(k)}-\theta_{t-1}^{(k)}\|_2\\
\leq&V\|\mathbf{f}_{t-1}\|_2\cdot\|\theta_t^{(k)}-\theta_{t-1}^{(k)}\|_2 + \|\mathbf{Q}(t)\|_2\cdot
\sqrt{\sum_{i=1}^m\|\mathbf{g}_{i,t-1}^{(k)}\|_2^2}\cdot\|\theta_t^{(k)}-\theta_{t-1}^{(k)}\|_2,
\end{align*}
where the second and third inequality follow from Cauchy-Schwarz inequality. Thus, it follows
\[
\l\|\theta_t^{(k)}-\theta_{t-1}^{(k)}\r\|_2
\leq\frac{V\|\mathbf{f}_{t-1}^{(k)}\|_2+\|\mathbf{Q}(t)\|_2\cdot
\sqrt{\sum_{i=1}^m\|\mathbf{g}_{i,t-1}^{(k)}\|_2^2}}{2\alpha}.
\]
Applying the fact that $\|\mathbf{f}_{t-1}^{(k)}\|_2\leq 
\sqrt{|\mathcal{A}^{(k)}||\mathcal{S}^{(k)}|}\Psi$, $\|\mathbf{g}_{i,t-1}^{(k)}\|_2\leq \sqrt{|\mathcal{A}^{(k)}||\mathcal{S}^{(k)}|}\Psi$ and taking expectation from both sides give the first bound in the lemma. The second bound follows directly from the first bound by further substituting \eqref{expected-Q-bound}.
\end{proof}

\begin{theorem}\label{thm:constraint-violation}
The proposed algorithm has the following stationary state constraint violation bound:
\begin{multline*}
\frac1T\sum_{t=0}^{T-1}\expect{\sum_{k=1}^K\l\langle \mathbf{g}_{i,t}^{(k)},\theta_t^{(k)} \r\rangle} \leq \frac{1}{\sqrt{T}}\l(C+\sum_{k=1}^K\sqrt{m|\mathcal{A}^{(k)}||\mathcal{S}^{(k)}|}\Psi C
+\sum_{k=1}^K|\mathcal{A}^{(k)}||\mathcal{S}^{(k)}|\Psi^2\r),
\end{multline*}
where $C=C(m,K,\Psi,\eta)$ is defined in \eqref{expected-Q-bound}.
\end{theorem}
\begin{proof}
Taking expectation from both sides of Lemma \ref{lemma:pre-Q-bound} gives
$$\sum_{t=1}^{T}\expect{\sum_{k=1}^K\l\langle \mathbf{g}_{i,t-1}^{(k)}, \theta_{t-1}^{(k)} \r\rangle }
\leq \expect{Q_i(T+1)} +  \Psi\sum_{t=1}^T\sum_{k=1}^K\sqrt{\l| \mathcal{A}^{(k)} \r|\l| \mathcal{S}^{(k)} \r|}  \expect{\l\| \theta^{(k)}_t - \theta^{(k)}_{t-1} \r\|_2}.$$
Substituting the bounds \eqref{expected-Q-bound} and \eqref{one-step-update} in to the above inequality gives the desired result.
\end{proof}

\subsection{Markov analysis}\label{sec:markov}
So far, we have shown that our algorithm achieves an $\mathcal{O}(\sqrt{T})$ regret and constraint violation simultaneously regarding the stationary online linear program \eqref{stat-regret} with constraint set given by \eqref{stat-constraint} in the imaginary system.
In this section, we show how these stationary state results leads to a tight performance bound on the original true online MDP problem \eqref{main-regret} and \eqref{main-constraint} comparing to any joint randomized stationary algorithm starting from its stationary state.

\subsubsection{Approximate mixing of MDPs}
Let $\mathcal{F}_t,~t\geq1$ be the set of system history functions up to time $t$,  including $f^{(k)}_0,\cdots,f^{(k)}_{t-1}$, $g^{(k)}_{0,i},\cdots,g^{(k)}_{t-1,i}$, $i=1,2,\cdots,m,~k=1,2,\cdots,K$, and $\mathcal{F}_0$ is a null set. Let $d_{\pi_t^{(k)}}$ be the stationary state distribution at $k$-th MDP under the randomized stationary policy $\pi_t^{(k)}$ in the proposed algorithm. 
Let $v_t^{(k)}$ be the true state distribution at time slot $t$ under the proposed algorithm given the function path $\mathcal{F}_{T}$ \xcolor{and starting state $d_0^{(k)}$, i.e. for any $s\in\mathcal{S}^{(k)}$, $v_t^{(k)}(s):=Pr\l(s_t^{(k)}=s | \mathcal{F}_{T}\r)$ and $v_0^{(k)}=d_0^{(k)}$.}

The following lemma provides a key estimate on the distance between stationary distribution and true distribution at each time slot $t$. 
It builds upon the slow-update condition (Lemma \ref{lemma:slow-update}) of the proposed algorithm and uniform mixing bound of general MDPs 
(Lemma \ref{lemma:mixing1}).

\begin{lemma}\label{lemma:distance-dist}
\xcolor{
Consider the proposed algorithm with $V=\sqrt{T}$ and $\alpha=T$. 
For any initial state distribution $\{d_0^{(k)}\}_{k=1}^K$ and any $t\in\{0,1,2,\cdots,T-1\}$, we have}
\[
\expect{\l\| d_{\pi_t^{(k)}} - v_t^{(k)} \r\|_1  }\leq \frac{\tau r\l(\l|\mathcal{A}^{(k)}\r|\l|\mathcal{S}^{(k)}\r|\Psi  + C\sqrt{m} \l|\mathcal{A}^{(k)}\r|\l|\mathcal{S}^{(k)}\r|\Psi \r)}{2\sqrt{T}}
+2e^{-\frac{t}{\tau r}+1},
\]
where $\tau$ and $r$ are mixing parameters defined in Lemma \ref{lemma:mixing1} and $C$ is an absolute constant defined in \eqref{expected-Q-bound}.
\end{lemma}

\begin{proof}
By Lemma \ref{lemma:slow-update} we know that for any $t\in\{1,2,\cdots,T\}$,
\[
\expect{\l\| \theta_t^{(k)}-\theta_{t-1}^{(k)} \r\|_2}\leq \frac{\sqrt{\l|\mathcal{A}^{(k)}\r|\l|\mathcal{S}^{(k)}\r|}\Psi  + C\sqrt{m\l|\mathcal{A}^{(k)}\r|\l|\mathcal{S}^{(k)}\r|}\Psi}{2\sqrt{T}},
\]
Thus,
\[
\expect{\l\| \theta_t^{(k)}-\theta_{t-1}^{(k)} \r\|_1}\leq \frac{\l|\mathcal{A}^{(k)}\r|\l|\mathcal{S}^{(k)}\r|\Psi  + C\sqrt{m}\l|\mathcal{A}^{(k)}\r|\l|\mathcal{S}^{(k)}\r|\Psi}{2\sqrt{T}},
\]
Since for any $s\in\mathcal{S}^{(k)}$, 
$$\big|d_{\pi_t^{(k)}}(s)-d_{\pi_{t-1}^{(k)}}(s)\big| = \Big| \sum_{a\in\mathcal{A}^{(k)}}\theta_t^{(k)} (a,s)-\theta_{t-1}^{(k)}(a,s) \Big|
\leq  \sum_{a\in\mathcal{A}^{(k)}}\Big| \theta_t^{(k)} (a,s)-\theta_{t-1}^{(k)}(a,s) \Big|, $$
it then follows
\begin{equation}\label{slow-update-dist}
\expect{\l\| d_{\pi_t^{(k)}}-d_{\pi_{t-1}^{(k)}} \r\|_1}\leq \expect{\l\| \theta_t^{(k)}-\theta_{t-1}^{(k)} \r\|_1}\leq
\frac{\l|\mathcal{A}^{(k)}\r|\l|\mathcal{S}^{(k)}\r|\Psi  + C\sqrt{m}\l|\mathcal{A}^{(k)}\r|\l|\mathcal{S}^{(k)}\r|\Psi}{2\sqrt{T}}.
\end{equation}
Now, we use the above relation to bound $\expect{\l\| d_{\pi_t^{(k)}} - v_t^{(k)} \r\|_1  }$ for any $t\geq r$.
\begin{align}
&\expect{\l\| d_{\pi_t^{(k)}} - v_t^{(k)} \r\|_1  }  \nonumber\\
\leq& \expect{\l\| d_{\pi_t^{(k)}} - d_{\pi_{t-1}^{(k)}} \r\|_1  }  +  \expect{\l\| d_{\pi_{t-1}^{(k)}} - v_t^{(k)} \r\|_1  }  \nonumber\\
\leq&\frac{\l|\mathcal{A}^{(k)}\r|\l|\mathcal{S}^{(k)}\r|\Psi  + C\sqrt{m}\l|\mathcal{A}^{(k)}\r|\l|\mathcal{S}^{(k)}\r|\Psi}{2\sqrt{T}}
+ \expect{\l\| d_{\pi_{t-1}^{(k)}} - v_t^{(k)} \r\|_1  }   \nonumber\\
=& \frac{\l|\mathcal{A}^{(k)}\r|\l|\mathcal{S}^{(k)}\r|\Psi  + C\sqrt{m}\l|\mathcal{A}^{(k)}\r|\l|\mathcal{S}^{(k)}\r|\Psi}{2\sqrt{T}}
+ \expect{\l\| \l(d_{\pi_{t-1}^{(k)}} - v_{t-1}^{(k)}\r)\mathbf{P}_{\pi_{t-1}^{(k)}}^{(k)} \r\|_1  }, \label{eq:procedure}
\end{align}
where the second inequality follows from the slow-update condition \eqref{slow-update-dist} and the final equality follows from the fact that given the function path $\mathcal{F}_T$, the following holds
\begin{equation}\label{eq:transfer}
d_{\pi_{t-1}^{(k)}} - v_t^{(k)}  = \l(d_{\pi_{t-1}^{(k)}} - v_{t-1}^{(k)}\r)\mathbf{P}_{\pi^{(k)}_{t-1}}^{(k)}.
\end{equation}
To see this, note that from the proposed algorithm, the policy $\pi^{(k)}_t$ is determined by $\mathcal{F}_{T}$. Thus, by definition of stationary distribution, given $\mathcal{F}_{T}$, we know that $d_{\pi_{t-1}^{(k)}}= d_{\pi_{t-1}^{(k)}}\mathbf{P}_{\pi_{t-1}^{(k)}}^{(k)}$, and it is enough to show that given $\mathcal{F}_{T}$,
\begin{equation*}
v_t^{(k)} = v_{t-1}^{(k)}\mathbf{P}_{\pi_{t-1}^{(k)}}^{(k)}.
\end{equation*}
First of all, the state distribution $v_t^{(k)}$ is determined by $v_{t-1}^{(k)}$, $\pi_{t-1}^{(k)}$ and probability transition from $s_{t-1}$ to $s_t$, which are in turn determined by
$\mathcal{F}_{T}$. Thus, given $\mathcal{F}_{T}$,
for any $s\in\mathcal{S}^{(k)}$, 
\begin{align*}
v_t^{(k)}(s) = \sum_{s'\in\mathcal{S}^{(k)}}Pr(s_t = s | s_{t-1} = s', \mathcal{F}_{T}) v_{t-1}^{(k)}(s'),
\end{align*}
and 
\begin{multline*}
Pr(s_{t}=s | s_{t-1}=s', \mathcal{F}_{T}) = \sum_{a\in\mathcal{A}^{(k)}}Pr(s_t=s | a_t=a, s_{t-1}=s', \mathcal{F}_{T})Pr(a_t=a| s_{t-1}=s', \mathcal{F}_{T})\\
= \sum_{a\in\mathcal{A}^{(k)}}P_a(s',s)Pr(a_t= a| s_{t-1}=s', \mathcal{F}_{T})
= \sum_{a\in\mathcal{A}^{(k)}}P_a(s',s)\pi^{(k)}_{t-1}(a|s') = P_{\pi_{t-1}^{(k)}}(s',s),
\end{multline*}
where the second inequality follows from the Assumption \ref{assumption:indep-trans}, the third equality follows from the fact that $\pi^{(k)}_{t-1}$ is determined by $\mathcal{F}_{T}$,
thus, for any $t$,
\begin{equation*}
\pi_{t}^{(k)}(a \big| s') = Pr(a_t=a | s_{t-1}=s', \mathcal{F}_{T}),~\forall a\in\mathcal{A}^{(k)},~s'\in\mathcal{S}^{(k)},
\end{equation*}
and the last equality follows from the definition of transition probability \eqref{transition-matrix}. This gives
\[
v_t^{(k)}(s) = \sum_{s'\in\mathcal{S}^{(k)}}P_{\pi_{t-1}^{(k)}}(s',s) v_{t-1}^{(k)}(s'),
\]
and thus \eqref{eq:transfer} holds.

We can iteratively apply the procedure \eqref{eq:procedure} $r$ times as follows
\begin{align*}
&\expect{\l\| d_{\pi_t^{(k)}} - v_t^{(k)} \r\|_1  }\\
\leq& \frac{\l|\mathcal{A}^{(k)}\r|\l|\mathcal{S}^{(k)}\r|\Psi  + C\sqrt{m}\l|\mathcal{A}^{(k)}\r|\l|\mathcal{S}^{(k)}\r|\Psi}{2\sqrt{T}}
 + \expect{\l\| \l(d_{\pi_{t-1}^{(k)}} - d_{\pi_{t-2}^{(k)}}\r)\mathbf{P}_{\pi_{t-1}^{(k)}}^{(k)} \r\|_1  }
 + \expect{\l\| \l(d_{\pi_{t-2}^{(k)}} - v_{t-1}^{(k)}\r)\mathbf{P}_{\pi_{t-1}^{(k)}}^{(k)} \r\|_1  }\\
 \leq& 2\cdot\frac{\l|\mathcal{A}^{(k)}\r|\l|\mathcal{S}^{(k)}\r|\Psi  + C\sqrt{m}\l|\mathcal{A}^{(k)}\r|\l|\mathcal{S}^{(k)}\r|\Psi}{2\sqrt{T}}
 + \expect{\l\| \l(d_{\pi_{t-2}^{(k)}} - v_{t-1}^{(k)}\r)\mathbf{P}_{\pi_{t-1}^{(k)}}^{(k)} \r\|_1  }\\
 \leq&2\cdot\frac{\l|\mathcal{A}^{(k)}\r|\l|\mathcal{S}^{(k)}\r|\Psi  + C\sqrt{m}\l|\mathcal{A}^{(k)}\r|\l|\mathcal{S}^{(k)}\r|\Psi}{2\sqrt{T}}
 + \expect{\l\| \l(d_{\pi_{t-2}^{(k)}} - v_{t-2}^{(k)}\r)\mathbf{P}_{\pi_{t-2}^{(k)}}^{(k)}\mathbf{P}_{\pi_{t-1}^{(k)}}^{(k)} \r\|_1  }\\
 \leq&\cdots
 \leq r\cdot\frac{\l|\mathcal{A}^{(k)}\r|\l|\mathcal{S}^{(k)}\r|\Psi  + C\sqrt{m}\l|\mathcal{A}^{(k)}\r|\l|\mathcal{S}^{(k)}\r|\Psi}{2\sqrt{T}}
 + \expect{\l\| \l(d_{\pi_{t-r_k}^{(k)}} - v_{t-r_k}^{(k)}\r)\mathbf{P}_{\pi_{t-r}^{(k)}}^{(k)}\cdots\mathbf{P}_{\pi_{t-1}^{(k)}}^{(k)} \r\|_1  },
\end{align*} 
where the second inequality follows from the nonexpansive property in $\ell_1$ norm\footnote{For an one-line proof, see \eqref{non-expansive} in the appendix.} of the stochastic matrix $\mathbf{P}_{\pi_{t-1}^{(k)}}^{(k)}$ that
\[
\l\| \l(d_{\pi_{t-1}^{(k)}} - d_{\pi_{t-2}^{(k)}}\r)\mathbf{P}_{\pi_{t-1}^{(k)}}^{(k)} \r\|_1\leq \l\| d_{\pi_{t-1}^{(k)}} - d_{\pi_{t-2}^{(k)}}\r\|_1,
\]
and then using the slow-update condition \eqref{slow-update-dist} again.
By Lemma \ref{lemma:mixing1}, we have
\[
\expect{\l\| d_{\pi_t^{(k)}} - v_t^{(k)} \r\|_1  }\leq r\cdot\frac{\l|\mathcal{A}^{(k)}\r|\l|\mathcal{S}^{(k)}\r|\Psi  + C\sqrt{m}\l|\mathcal{A}^{(k)}\r|\l|\mathcal{S}^{(k)}\r|\Psi}{2\sqrt{T}}
+e^{-1/\tau}\expect{\l\| d_{\pi_{t-r}^{(k)}} - v_{t-r}^{(k)} \r\|_1}.
\]
Iterating this inequality down to $t=0$ gives
\begin{align*}
\expect{\l\| d_{\pi_t^{(k)}} - v_t^{(k)} \r\|_1  }
\leq&
\sum_{j=0}^{\lfloor t/\tau\rfloor} e^{-j/\tau} \cdot r\cdot\frac{\l|\mathcal{A}^{(k)}\r|\l|\mathcal{S}^{(k)}\r|\Psi  + C\sqrt{m}\l|\mathcal{A}^{(k)}\r|\l|\mathcal{S}^{(k)}\r|\Psi}{2\sqrt{T}}
+ \expect{\l\|d_{\pi_0^{(k)}} - v_0^{(k)}\r\|_1}e^{-\lfloor t/r\rfloor/\tau}\\
\leq&
\sum_{j=0}^{\lfloor t/\tau\rfloor} e^{-j/\tau} \cdot r\cdot\frac{\l|\mathcal{A}^{(k)}\r|\l|\mathcal{S}^{(k)}\r|\Psi  + C\sqrt{m}\l|\mathcal{A}^{(k)}\r|\l|\mathcal{S}^{(k)}\r|\Psi}{2\sqrt{T}}
+ 2e^{-\lfloor t/r\rfloor/\tau}\\
\leq& \int_{x=0}^\infty e^{-x/\tau}dx\cdot r\cdot\frac{\l|\mathcal{A}^{(k)}\r|\l|\mathcal{S}^{(k)}\r|\Psi  + C\sqrt{m}\l|\mathcal{A}^{(k)}\r|\l|\mathcal{S}^{(k)}\r|\Psi}{2\sqrt{T}}
+ 2e^{- \frac{t}{r\tau}+1}\\
\leq&\tau r\cdot\frac{\l|\mathcal{A}^{(k)}\r|\l|\mathcal{S}^{(k)}\r|\Psi  + C\sqrt{m}\l|\mathcal{A}^{(k)}\r|\l|\mathcal{S}^{(k)}\r|\Psi}{2\sqrt{T}}
+ 2e^{- \frac{t}{r\tau}+1}
\end{align*}
finishing the proof.
\end{proof}

\subsubsection{Benchmarking against policies starting from stationary state}
Combining the results derived so far, we have the following regret bound regarding any randomized stationary policy $\Pi$ starting from its stationary state distribution $d_\Pi$ such that 
$(d_\Pi,\Pi)$ in the constraint set $\mathcal{G}$ defined in \eqref{main-constraint}.
\begin{theorem}\label{thm:final-1}
Let $\mathscr{P}$ be the sequence of randomized stationary policies resulting from the proposed algorithm with $V=\sqrt{T}$ and $\alpha=T$. Let $d_0$ be the staring state of the proposed algorithm.
For any randomized stationary policy $\Pi$ starting from its stationary state distribution $d_\Pi$ such that 
$(d_\Pi,\Pi)\in\mathcal{G}$, we have
\begin{align*}
&F_T(d_0,\mathscr{P}) - F_T(d_\Pi,\Pi)\leq \mathcal{O}\l(m^{3/2}K^2\sum_{k=1}^K\l|\mathcal{A}^{(k)}\r|\l|\mathcal{S}^{(k)}\r|\cdot\sqrt{T}\r),\\
&G_{i,T}(d_0,\mathscr{P})\leq \mathcal{O}\l(m^{3/2}K^2\sum_{k=1}^K\l|\mathcal{A}^{(k)}\r|\l|\mathcal{S}^{(k)}\r|\cdot\sqrt{T}\r),~i=1,2,\cdots,m.
\end{align*}
\end{theorem}

\begin{proof}
First of all, by Lemma \ref{lemma:prod-chain}, for any randomized stationary policy $\Pi$, there exists some stationary state-action probability vectors $\{\theta_*^{(k)}\}_{k=1}^K$ such that $\theta_*^{(k)}\in\Theta^{(k)}$, $F_T(d_\Pi,\Pi) = \sum_{t=0}^{T-1}\sum_{k=1}^K\l\langle\expect{\mathbf{f}_t},\theta^{(k)}_*\r\rangle$, and $G_{i,T}(d_\Pi,\Pi) = \sum_{t=0}^{T-1}\sum_{k=1}^K\l\langle\expect{\mathbf{g}_{i,t}},\theta^{(k)}_*\r\rangle$. As a consequence, $(d_\Pi,\Pi)\in\mathcal{G}$ implies $G_{i,T}(d_\Pi,\Pi) = \sum_{t=0}^{T-1}\sum_{k=1}^K\l\langle\expect{\mathbf{g}_{i,t}},\theta^{(k)}_*\r\rangle\leq0,~\forall i\in\{1,2,\cdots,m\}$ and it follows $\{\theta_*^{(k)}\}_{k=1}^K$ is in the imaginary constraint set $\overline{\mathcal{G}}$ defined in \eqref{stat-constraint}. Thus, we are in a good shape applying Theorem \ref{thm:stationary-regret} from imaginary systems.

We then split $F_T(d_0,\mathscr{P}) - F_T(d_\Pi,\Pi)$ into two terms:
\begin{align*}
F_T(d_0,\mathscr{P}) - F_T(d_0,\Pi)
\leq&
\underbrace{\left|\expect{\left.\sum_{t=0}^{T-1}\sum_{k=1}^Kf^{(k)}_t(a_t^{(k)},s_t^{(k)})\right|~d_0,\mathscr{P}} 
- \sum_{t=0}^{T-1}\sum_{k=1}^K\expect{\l\langle\mathbf{f}_t^{(k)},\theta_t^{(k)}\r\rangle}\right|}_{\text{(I)}}\\
&+ \underbrace{ \sum_{t=0}^{T-1}\sum_{k=1}^K\l(\expect{\l\langle\mathbf{f}_t^{(k)},\theta_t^{(k)}\r\rangle} - \l\langle\expect{\mathbf{f}_t},\theta^{(k)}_*\r\rangle\r) }_{\text{(II)}}.
\end{align*}
By Theorem \ref{thm:stationary-regret}, we get 
\begin{equation}\label{final-1}
(\text{II})\leq \l(2K + \frac{\Psi^2}{2}\sum_{k=1}^K\l|\mathcal{S}^{(k)}\r|\l|\mathcal{A}^{(k)}\r|+ \frac52mK^2\Psi^2\r)\sqrt{T}.
\end{equation}

We then bound (I). Consider each time slot $t\in\{0,1,\cdots,T-1\}$.  We have
\begin{align*}
\expect{\l\langle\mathbf{f}_t^{(k)},\theta_t^{(k)}\r\rangle }
=&\sum_{s\in\mathcal{S}^{(k)}}\sum_{a\in\mathcal{A}^{(k)}}\expect{ d_{\pi^{(k)}_t}(s) \pi^{(k)}_t(a|s)f_t^{(k)}(a,s)}\\
\expect{\left.f_t^{(k)}(a_t^{(k)},s_t^{(k)})\right|~d_0,\mathscr{P}}
=&\sum_{s\in\mathcal{S}^{(k)}} \sum_{a\in\mathcal{A}^{(k)}} \expect{v^{(k)}_t(s)\pi^{(k)}_t(a|s)f_t^{(k)}(a,s)},
\end{align*}
where the first equality follows from the definition of $\theta_t^{(k)}$ and the second equality follows from the following: Given a specific function path $\mathcal{F}_T$, the policy 
$\pi_t^{(k)}$ and the true state distribution $v_t^{(k)}$ are fixed. Thus, we have,
\[
\expect{\left.f_t^{(k)}(a_t^{(k)},s_t^{(k)})\right|~d_0,\mathscr{P},\mathcal{F}_T} = \sum_{s\in\mathcal{S}^{(k)}} \sum_{a\in\mathcal{A}^{(k)}}v^{(k)}_t(s)\pi^{(k)}_t(a|s)f_t^{(k)}(a,s).
\]
Taking the full expectation regarding the function path gives the result.
Thus,
\begin{align*}
&\left|\expect{\left.f_t^{(k)}(a_t^{(k)},s_t^{(k)})\right|~d_0,\mathscr{P}}-\expect{\l\langle\mathbf{f}_t^{(k)},\theta_t^{(k)}\r\rangle}\right|\\
\leq&\left| \sum_{s\in\mathcal{S}^{(k)}}\sum_{a\in\mathcal{A}^{(k)}}\expect{ \l(v_t^{(k)}(s) - d_{\pi_t^{(k)}}(s)\r)\pi^{(k)}_t(a|s) } \right| \Psi\\
\leq& \expect{\left\|  v_t^{(k)} - d_{\pi_t^{(k)}} \right\|_1} \Psi\\
\leq& \frac{\tau r\l(1 + C\sqrt{m}\r)\l|\mathcal{A}^{(k)}\r|\l|\mathcal{S}^{(k)}\r|\Psi^2 }{2\sqrt{T}}
+2e^{-\frac{t}{\tau r}+1}\Psi
\end{align*}
where the last inequality follows from Lemma \ref{lemma:distance-dist}.
Thus, it follows,

\begin{align}
\text{(I)}\leq& \sum_{t=0}^{T-1}\sum_{k=1}^K\l(\frac{\tau r\l(1 + C\sqrt{m}\r)\l|\mathcal{A}^{(k)}\r|\l|\mathcal{S}^{(k)}\r|\Psi^2 }{2\sqrt{T}}
+2e^{-\frac{t}{\tau r}+1}\Psi\r)   \nonumber\\
\leq& \sum_{k=1}^K\l(\tau r\l(1 + C\sqrt{m}\r)\l|\mathcal{A}^{(k)}\r|\l|\mathcal{S}^{(k)}\r|\Psi^2 \r) \sqrt{T}
+ 2\Psi K\int_{t=0}^{T-1}e^{-\frac{x}{\tau r}+1}dx    \nonumber\\
\leq& \tau r\Psi^2\l(1 + C\sqrt{m}\r)\sum_{k=1}^K\l|\mathcal{A}^{(k)}\r|\l|\mathcal{S}^{(k)}\r| \cdot \sqrt{T}
+2e\Psi K\tau r. \label{final-2}
\end{align}

Overall, combining \eqref{final-1},\eqref{final-2} and substituting the constant $C = C(m,K,\Psi,\eta)$ defined in \eqref{expected-Q-bound} 
gives the objective regret bound.

For the constraint violation, we have
\begin{equation*}
G_{i,T}(d_0,\mathscr{P}) = \underbrace{\expect{\left.\sum_{t=0}^{T-1}\sum_{k=1}^Kg_{i,t}^{(k)}(a_t,s_t)\right|~d_0,\mathscr{P}} 
- \sum_{t=1}^T\sum_{k=1}^K\l\langle\expect{\mathbf{g}_{i,t}^{(k)}},\theta_t\r\rangle}_{\text{(IV)}} 
+ \underbrace{\sum_{t=1}^T\sum_{k=1}^K\l\langle\expect{\mathbf{g}_{i,t}^{(k)}},\theta_t\r\rangle}_{\text{(V)}}.
\end{equation*}
The term (V) can be readily bounded using Theorem \ref{thm:constraint-violation} as 
\[
\sum_{t=0}^{T-1}\expect{\sum_{k=1}^K\l\langle \mathbf{g}_{i,t}^{(k)},\theta_t^{(k)} \r\rangle} \leq \l(C+\sum_{k=1}^K\sqrt{m|\mathcal{A}^{(k)}||\mathcal{S}^{(k)}|}\Psi C
+\sum_{k=1}^K|\mathcal{A}^{(k)}||\mathcal{S}^{(k)}|\Psi^2\r)\sqrt{T}.
\]
For the term (IV), we have
\begin{align*}
\expect{\l\langle\mathbf{g}_{i,t}^{(k)},\theta_t^{(k)}\r\rangle }
=&\sum_{s\in\mathcal{S}^{(k)}}\sum_{a\in\mathcal{A}^{(k)}}\expect{ d_{\pi^{(k)}_t}(s) \pi^{(k)}_t(a|s)g_{i,t}^{(k)}(a,s)}\\
\expect{\left.g_{i,t}^{(k)}(a_t^{(k)},s_t^{(k)})\right|~d_0,\mathscr{P}}
=&\sum_{s\in\mathcal{S}^{(k)}} \sum_{a\in\mathcal{A}^{(k)}} \expect{v^{(k)}_t(s)\pi^{(k)}_t(a|s)g_{i,t}^{(k)}(a,s)},
\end{align*}
where the first equality follows from the definition of $\theta_t^{(k)}$ and the second equality follows from the following:
Given a specific function path $\mathcal{F}_T$, the policy $\pi_t^{(k)}$ and the true state distribution $v_t^{(k)}$ are fixed. Thus, we have,
\[
\expect{\left.g_t^{(k)}(a_t^{(k)},s_t^{(k)})\right|~d_0,\mathscr{P},\mathcal{F}_T} = \sum_{s\in\mathcal{S}^{(k)}} \sum_{a\in\mathcal{A}^{(k)}}v^{(k)}_t(s)\pi^{(k)}_t(a|s)g_t^{(k)}(a,s).
\]
Taking the full expectation regarding the function path gives the result.
Then, repeat the same proof as that of \eqref{final-2} gives
\[
\text{(IV)}\leq \tau r\Psi^2\l(1 + C\sqrt{m}\r)\sum_{k=1}^K\l|\mathcal{A}^{(k)}\r|\l|\mathcal{S}^{(k)}\r| \cdot \sqrt{T}
+2e\Psi K\tau r. 
\]
This finishes the proof of constraint violation.
\end{proof}

\section{A more general regret bound against policies with arbitrary starting state}\label{sec:perturb}
Recall that Theorem \ref{thm:final-1} compares the proposed algorithm with any randomized stationary policy $\Pi$ starting from its stationary state distribution $d_\Pi$, so that $(d_\Pi,\Pi)\in\mathcal{G}$. In this section, we generalize Theorem \ref{thm:final-1} and obtain a bound of the regret against all $(d_0,\Pi)\in\mathcal{G}$ where $d_0$ is an arbitrary starting state distribution (not necessarily the stationary state distribution). The main technical difficulty doing such a generalization is as follows: For any randomized stationary policy $\Pi$ such that $(d_0,\Pi)\in\mathcal{G}$, let  
$\{\theta_*^{(k)}\}_{k=1}^K$ be the stationary state-action probabilities such that $\theta_*^{(k)}\in\Theta^{(k)}$ and $G_{i,T}(d_\Pi,\Pi) = \sum_{t=0}^{T-1}\sum_{k=1}^K\l\langle\expect{\mathbf{g}_{i,t}},\theta^{(k)}_*\r\rangle$. For some finite horizon $T$, there might exist some ``low-cost" starting state distribution $d_0$ such that 
$G_{i,T}(d_0,\Pi) < G_{i,T}(d_\Pi,\Pi)$ for some $i\in\{1,2,\cdots,m\}$. As a consequence, one coud have 
\[
G_{i,T}(d_0,\Pi) \leq 0,~\text{and}~\sum_{t=0}^{T-1}\sum_{k=1}^K\l\langle\expect{\mathbf{g}_{i,t}},\theta^{(k)}_*\r\rangle>0.
\]
This implies although $(d_0,\Pi)$ is feasible for our true system, its stationary state-action probabilities $\{\theta_*^{(k)}\}_{k=1}^K$ can be \textit{infeasible} with respect to the imaginary constraint set \eqref{stat-constraint}, and all our analysis so far fails to cover such randomized stationary policies.

To resolve this issue, we have to ``enlarge'' the imaginary constraint set \eqref{stat-constraint} so as to cover all state-action probabilities $\{\theta_*^{(k)}\}_{k=1}^K$ arising from any randomized stationary policy $\Pi$ such that $(d_0,\Pi)\in\mathcal{G}$. But a perturbation of constraint set would result in a perturbation of objective in the imaginary system also. Our main goal in this section is to bound such a perturbation and show that the perturbation bound leads to the final $\mathcal{O}(\sqrt{T})$ regret bound.

\subsubsection{A relaxed constraint sets}
We begin with a supporting lemma on the uniform mixing time bound over all joint randomized stationary policies. The proof is given in the appendix.
\begin{lemma}\label{lemma:small-bound}
Consider any randomized stationary policies $\Pi$ in \eqref{main-constraint} with arbitrary starting state distribution $d_0\in\mathcal{S}^{(1)}\times\cdots\times\mathcal{S}^{(K)}$. Let $\mathbf{P}_\Pi$ be the corresponding transition matrix on the product state space. Then, the following holds
\begin{equation}\label{joint-mixing}
\l\| (d_0-d_\Pi)\l(\mathbf{P}_{\Pi}\r)^t \r\|_1 \leq 2e^{(r_1-t)/r_1}, \forall t\in\{0,1,2,\cdots\},
\end{equation}
where $r_1$ is fixed positive constant independent of $\Pi$.
\end{lemma}

The following lemma shows a relaxation of $\mathcal{O}(1/T)$ on the imaginary constraint set \eqref{stat-constraint} is enough to cover all the $\{\theta_*^{(k)}\}_{k=1}^K$ discussed at the beginning of this section.

\begin{lemma}\label{lemma:perturbed}
For any $T\in\{1,2,\cdots\}$ and any randomized stationary policies $\Pi$ in \eqref{main-constraint}, with arbitrary starting state distribution $d_0\in\mathcal{S}^{(1)}\times\cdots\times\mathcal{S}^{(K)}$ and stationary state-action probability $\{\theta_*^{(k)}\}_{k=1}^K$,
\begin{align}
&\sum_{t=0}^{T-1}\l| \expect{\sum_{k=1}^K f_{t}^{(k)}(a^{(k)}_t,s^{(k)}_t) \Big| d_0,\Pi }  -  
\sum_{k=1}^K\l\langle \expect{\mathbf{f}^{(k)}_{t}}, \theta_*^{(k)} \r\rangle  \r| \leq  C_1K\Psi  \label{diff-1}\\
& \sum_{t=0}^{T-1}\l| \expect{\sum_{k=1}^K g_{i,t}^{(k)}(a^{(k)}_t,s^{(k)}_t) \Big| d_0,\Pi }  -  
\sum_{k=1}^K\l\langle \expect{\mathbf{g}^{(k)}_{i,t}}, \theta_*^{(k)} \r\rangle  \r| \leq C_1K\Psi \label{diff-2}
\end{align}
where $C_1$ is an absolute constant. In particular, $\{\theta_*^{(k)}\}_{k=1}^K$ is contained in the following 
relaxed constraint set 
\[
\overline{\mathcal{G}}^+:=\l\{ \theta^{(k)}\in\Theta^{(k)},~k=1,2,\cdots,K:~\sum_{k=1}^K\l\langle \expect{\mathbf{g}^{(k)}_{i,t}}, \theta^{(k)} \r\rangle\leq\frac{C_1K\Psi}{T}
,i=1,2,\cdots,m  \r\},
\]
for some universal positive constant $r_1>0$.
\end{lemma}

\begin{proof}
Let $v_t\in\mathcal{S}^{(1)}\times\cdots\times\mathcal{S}^{(K)}$ be the joint state distribution at time $t$ under policy $\Pi$.
Using the fact that $\Pi$ is a fixed policy independent of  $\mathbf{g}^{(k)}_{i,t}$ and Assumption \ref{assumption:indep-trans} that the probability transition is also independent of function path given any state and action,
 the function $\mathbf{g}^{(k)}_{i,t}$ and state-action pair 
$(a^{(k)}_t,s^{(k)}_t)$ are mutually independent.
Thus, for any $t\in\{0,1,2,\cdots,T-1\}$
\[
 \expect{\sum_{k=1}^K g_{i,t}^{(k)}(a^{(k)}_t,s^{(k)}_t) \Big| d_0,\Pi }=\sum_{\mathbf{s}\in\mathcal{S}^{(1)}\times\cdots\times\mathcal{S}^{(K)}}
 \sum_{\mathbf{a}\in\mathcal{A}^{(1)}\times\cdots\times\mathcal{A}^{(K)}}v_t(\mathbf{s})\Pi(\mathbf{a}|\mathbf{s})\sum_{k=1}^K\expect{g_{i,t}^{(k)}(a^{(k)},s^{(k)})},
\]
where $\mathbf{s}=[s^{(1)},\cdots,s^{(K)}]$ and $\mathbf{a}=[a^{(1)},\cdots,a^{(K)}]$ and the latter expectation is taken with respect to $\mathbf{g}_{i,t}^{(k)}$ (i.e. the random variable $w_{t}$).
On the other hand, by Lemma \ref{lemma:prod-chain}, we know that for any randomized stationary policy $\Pi$, the corresponding stationary state-action probability can be expressed as 
$\{\theta_*^{(k)}\}_{k=1}^K$ with $\theta_*^{(k)}\in\Theta^{(k)}$. Thus,
\[
\sum_{k=1}^K\l\langle \expect{\mathbf{g}^{(k)}_{i,t}}, \theta^{(k)} \r\rangle
= \sum_{\mathbf{s}\in\mathcal{S}^{(1)}\times\cdots\times\mathcal{S}^{(K)}}\sum_{\mathbf{a}\in\mathcal{A}^{(1)}\times\cdots\times\mathcal{A}^{(K)}}d_\Pi(\mathbf{s})\Pi(\mathbf{a}|\mathbf{s})\sum_{k=1}^K\expect{g_{i,t}^{(k)}(a^{(k)},s^{(k)})}.
\]
Hence,
we can control the difference:
\begin{align*}
&   \sum_{t=0}^{T-1}\l| \expect{\sum_{k=1}^K g_{i,t}^{(k)}(a^{(k)}_t,s^{(k)}_t) \Big| d_0,\Pi }  -  
\sum_{k=1}^K\l\langle \expect{\mathbf{g}^{(k)}_{i,t}}, \theta_*^{(k)} \r\rangle  \r|   \\
\leq& \sum_{t=0}^{T-1}\l| \sum_{\mathbf{s}\in\mathcal{S}^{(1)}\times\cdots\times\mathcal{S}^{(K)}} \sum_{\mathbf{a}\in\mathcal{A}^{(1)}\times\cdots\times\mathcal{A}^{(K)}} 
\l( v_t(\mathbf{s}) - d_\Pi(\mathbf{s}) \r)\Pi(\mathbf{a}|\mathbf{s})
 \r| K\Psi    \\
\leq& K\Psi\sum_{t=0}^{T-1}\|v_t-d_\Pi\|_1\leq    2K\Psi\sum_{t=0}^{T-1}e^{(r_1-t)/r_1}
\leq 2eK\Psi\int_0^{T-1}e^{-t/r_1}dt
= 2er_1K\Psi,
\end{align*}
where the third inequality follows from Lemma \ref{lemma:small-bound}.
Taking $C_1 = 2er_1$ finishes the proof of \eqref{diff-2} and \eqref{diff-1} can be proved in a similar way.

In particular, we have for any randomized stationary policy $\Pi$ that satisfies the constraint \eqref{main-constraint}, we have
\begin{align*}
&  T\cdot\sum_{k=1}^K\l\langle \expect{\mathbf{g}^{(k)}_{i,t}}, \theta_*^{(k)} \r\rangle\\
\leq&\sum_{t=0}^{T-1}\l| \expect{\sum_{k=1}^K g_{i,t}^{(k)}(a^{(k)}_t,s^{(k)}_t) \Big| d_0,\Pi }  -  
\sum_{k=1}^K\l\langle \expect{\mathbf{g}^{(k)}_{i,t}}, \theta_*^{(k)} \r\rangle  \r|  
+  \sum_{t=0}^{T-1}\expect{\sum_{k=1}^K g_{i,t}^{(k)}(a^{(k)}_t,s^{(k)}_t) \Big| d_0,\Pi }\\
\leq&2er_1K\Psi + 0 = 2er_1K\Psi,
\end{align*}
finishing the proof.
\end{proof}

\subsubsection{Best stationary performance over the relaxed constraint set}
Recall that the best stationary performance in hindsight over all randomized stationary policies in the constraint set $\overline{\mathcal{G}}$ can be obtained as the minimum achieved by the following linear program.

\begin{align}
\min&~~\frac1T\sum_{t=0}^{T-1}\sum_{k=1}^K\l\langle \expect{\mathbf{f}_t^{(k)}}, \theta^{(k)}  \r\rangle   \label{ori-lp1}\\
s.t.&~~ \sum_{k=1}^K\l\langle \expect{\mathbf{g}_{i,t}^{(k)}}, \theta^{(k)}  \r\rangle\leq 0,~~i=1,2,\cdots,m.  \label{ori-lp2}
\end{align}

On the other hand, if we consider all the randomized stationary policies contained in the original constraint set \eqref{main-constraint}, then, By Lemma \ref{lemma:perturbed}, the relaxed constraint set $\overline{\mathcal{G}}$ contains all such policies and 
the best stationary performance over this relaxed set comes from the minimum achieved by the following perturbed linear program:
\begin{align}
\min&~~\frac1T\sum_{t=0}^{T-1}\sum_{k=1}^K\l\langle \expect{\mathbf{f}_t^{(k)}}, \theta^{(k)}  \r\rangle  \label{relax-lp1}\\
s.t.&~~ \sum_{k=1}^K\l\langle \expect{\mathbf{g}_{i,t}^{(k)}}, \theta^{(k)}  \r\rangle\leq \frac{C_1K\Psi}{T},~~i=1,2,\cdots,m.  \label{relax-lp2}
\end{align}

We aim to show that the minimum achieved by \eqref{relax-lp1}-\eqref{relax-lp2} is not far away from that of \eqref{ori-lp1}-\eqref{ori-lp2}. In general, such a conclusion is not true due to the unboundedness of Lagrange multipliers in constrained optimization. However, since Slater's condition holds in our case, the perturbation can be bounded via the following well-known Farkas' lemma (\cite{bertsekas2009convex}):

\begin{lemma}[Farkas' Lemma]\label{lemma:Farkas}
Consider a convex program with objective $f(x)$ and constraint function $g_i(x),~i=1,2,\cdots,m$:
\begin{align}
\min&~~f(x), \label{cp:1}   \\
s.t.&~~g_i(x)\leq b_i,~~i=1,2,\cdots,m,\\
&~~x\in\mathcal{X},  \label{cp:3}
\end{align}
for some convex set $\mathcal{X}\subseteq\mathbb{R}^n$. Let $x^*$ be one of the solutions to the above convex program. Suppose there exists $\widetilde{x}\in\mathcal{X}$ such that $g_i\l(\widetilde{x}\r)<0,~\forall i\in\{1,2,\cdots,m\}$. Then, there exists a separation hyperplane parametrized by $(1,\mu_1,\mu_2,\cdots,\mu_m)$ such that $\mu_i\geq0$ and
\[
f(x) + \sum_{i=1}^m\mu_ig_i(x)\geq f(x^*) + \sum_{i=1}^m\mu_ib_i,~~\forall x\in\mathcal{X}.
\]
\end{lemma}
The parameter $\mu=(\mu_1,\mu_2,\cdots,\mu_m)$ is usually referred to as a Lagrange multiplier.
From the geometric perspective, Farkas' Lemma states that if Slater's condition holds, then, there exists a non-vertical separation hyperplane supported at $\Big(f(x^*),b_1,\cdots,b_m\Big)$ and contains the set 
$\l\{\Big(f(x),g_1(x),\cdots,g_m(x)\Big),~x\in\mathcal{X}\r\}$ on one side. Thus, in order to bound the perturbation of objective with respect to the perturbation of constraint level, we need to bound the slope of the supporting hyperplane from above, which boils down to controlling the magnitude of the Lagrange multiplier.
This is summarized in the following lemma:

\begin{lemma}[Lemma 1 of \cite{nedic2009approximate}]\label{lemma:bound-lagrange}
Consider the convex program \eqref{cp:1}-\eqref{cp:3}, and define the Lagrange dual function $q(\mu)=\inf_{x\in\mathcal{X}}\l\{ f(x) + \sum_{i=1}^m\mu_i(g_i(x)-b_i) \r\}$. Suppose there exists $\widetilde{x}\in\mathcal{X}$ such that \xcolor{ $g_i\l(\widetilde{x}\r)-b_i\leq-\eta,~\forall i\in\{1,2,\cdots,m\}$} for some positive constant $\eta>0$. Then, the level set 
$\mathcal{V}_{\bar\mu}=\l\{ \mu_1,\mu_2,\cdots,\mu_m\geq0,~ q(\mu)\geq q(\bar\mu)\r\}$ is bounded for any nonnegative $\bar\mu$. Furthermore, we have
$\max_{\mu\in\mathcal{V}_{\bar\mu}}  \|\mu\|_2\leq\frac{1}{\min_{1\leq i \leq m}\l\{-g_i(\widetilde{x})+b_i\r\}}\l(  f(\widetilde{x})-q(\bar\mu)   \r)$.
\end{lemma}

The technical importance of these two lemmas in the current context is contained in the following corollary.
\begin{corollary}\label{coro:comp}
Let $\l\{\theta^{(k)}_*\r\}_{k=1}^K$ and $\l\{\overline\theta^{(k)}_*\r\}_{k=1}^K$ be solutions to \eqref{ori-lp1}-\eqref{ori-lp2} and \eqref{relax-lp1}-\eqref{relax-lp2}, respectively. Then, the following holds
\[
\frac1T\sum_{t=0}^{T-1}\sum_{k=1}^K\l\langle \expect{\mathbf{f}_t^{(k)}}, \overline\theta_*^{(k)}  \r\rangle
\geq \frac1T\sum_{t=0}^{T-1}\sum_{k=1}^K\l\langle \expect{\mathbf{f}^{(k)}}, \theta_*^{(k)}  \r\rangle - \frac{C_1K^2\sqrt{m}\Psi^2}{\eta T}
\]
\end{corollary}
\begin{proof}
Take 
\begin{align*}
&f\l( \theta^{(1)},\cdots,\theta^{(K)} \r) = \frac1T\sum_{t=0}^{T-1}\sum_{k=1}^K\l\langle \expect{\mathbf{f}^{(k)}}, \theta^{(k)}  \r\rangle,  \\
&g_i\l( \theta^{(1)},\cdots,\theta^{(K)} \r) = \sum_{k=1}^K\l\langle \expect{\mathbf{g}_{i,t}^{(k)}}, \theta^{(k)}  \r\rangle,\\
&\mathcal{X} = \Theta^{(1)}\times\Theta^{(2)}\times\cdots\times\Theta^{(K)},
\end{align*}
and $b_i=0$
in Farkas' Lemma and we have the following display
\begin{equation*}
\frac1T\sum_{t=0}^{T-1}\sum_{k=1}^K\l\langle \expect{\mathbf{f}^{(k)}}, \theta^{(k)}  \r\rangle
+ \sum_{i=1}^m\mu_i\sum_{k=1}^K\l\langle \expect{\mathbf{g}_{i,t}^{(k)}}, \theta^{(k)}  \r\rangle
\geq \frac1T\sum_{t=0}^{T-1}\sum_{k=1}^K\l\langle \expect{\mathbf{f}^{(k)}}, \theta_*^{(k)}  \r\rangle,
\end{equation*}
for any $\l( \theta^{(1)},\cdots,\theta^{(K)} \r)\in\mathcal{X}$ and some $\mu_1,\mu_2,\cdots,\mu_m\geq0$. In particular, substituting $\l( \overline\theta^{(1)}_*,\cdots,\overline\theta^{(K)}_* \r)$ into the above display gives
\begin{align}
\frac1T\sum_{t=0}^{T-1}\sum_{k=1}^K\l\langle \expect{\mathbf{f}^{(k)}}, \overline\theta^{(k)}_*  \r\rangle
\geq \frac1T\sum_{t=0}^{T-1}\sum_{k=1}^K\l\langle \expect{\mathbf{f}^{(k)}}, \theta_*^{(k)}  \r\rangle
-  \sum_{i=1}^m\mu_i\sum_{k=1}^K\l\langle \expect{\mathbf{g}_{i,t}^{(k)}}, \overline\theta_*^{(k)}  \r\rangle   \nonumber\\
\geq \frac1T\sum_{t=0}^{T-1}\sum_{k=1}^K\l\langle \expect{\mathbf{f}^{(k)}}, \theta_*^{(k)}  \r\rangle
-  \frac{C_1K\Psi}{T}\sum_{i=1}^m\mu_i,  \label{inter-bound}
\end{align}
where the final inequality follows from the fact that $\l( \overline\theta^{(1)}_*,\cdots,\overline\theta^{(K)}_* \r)$ satisfies the relaxed constraint 
$\sum_{k=1}^K\l\langle \expect{\mathbf{g}_{i,t}^{(k)}}, \overline\theta_*^{(k)}  \r\rangle\leq  \frac{C_1K\Psi}{T}$ and $\mu_i\geq0,~\forall i\in\{1,2,\cdots,m\}$. Now we need to bound the magnitude of Lagrange multiplier $\l(\mu_1,\cdots,\mu_m\r)$. Note that in our scenario, 
\[
\Big|f\l( \theta^{(1)},\cdots,\theta^{(K)} \r) \Big|=\l| \frac1T\sum_{t=0}^{T-1}\sum_{k=1}^K\l\langle \expect{\mathbf{f}^{(k)}}, \theta^{(k)}  \r\rangle  \r|\leq \Psi,
\]
and the Lagrange multiplier $\mu$ is the solution to the maximization problem 
$$\max_{\mu_i\geq0,i\in\{1,2,\cdots,m\}}q(\mu),$$ 
where $q(\mu)$ is the dual function defined in Lemma \ref{lemma:bound-lagrange}.
thus, it must be in any super level set $\mathcal{V}_{\bar\mu}=\l\{ \mu_1,\mu_2,\cdots,\mu_m\geq0,~ q(\mu)\geq q(\bar\mu)\r\}$.
In particular, taking $\bar\mu = 0$ in Lemma \ref{lemma:bound-lagrange} and using Slater's condition \eqref{slater-2}, we have
there exists $\widetilde{\theta}^{(1)},\cdots,\widetilde{\theta}^{(K)}$ such that
\begin{multline*}
\sum_{i=1}^m\mu_i \leq \sqrt{m} \|\mu\|_2\leq \frac{\sqrt{m}}{\mu}\l(f\l(\widetilde{\theta}^{(1)},\cdots,\widetilde{\theta}^{(K)}\r)
-\inf_{\l(\theta^{(1)},\cdots,\theta^{(K)}\r)\in\mathcal{X}}f\l( \theta^{(1)},\cdots,\theta^{(K)} \r)\r)
\leq \frac{2\sqrt{m}\Psi K}{\eta},
\end{multline*}
where the final inequality follows from the deterministic bound of $|f(\theta^{(1)},\cdots,\theta^{(K)})|$ by $\Psi K$.
Substituting this bound into \eqref{inter-bound} gives the desired result.
\end{proof}

As a simple consequence of the above corollary, we have our final bound on the regret and constraint violation regarding any $(d_0,\Pi)\in\mathcal{G}$.
\begin{theorem}\label{thm:final-regret}
Let $\mathscr{P}$ be the sequence of randomized stationary policies resulting from the proposed algorithm with $V=\sqrt{T}$ and $\alpha=T$. Let $d_0$ be the staring state of the proposed algorithm.
For any randomized stationary policy $\Pi$ starting from the state $d_0$ such that 
$(d_0,\Pi)\in\mathcal{G}$, we have
\begin{align*}
&F_T(d_0,\mathscr{P}) - F_T(d_0,\Pi)\leq \mathcal{O}\l(m^{3/2}K^2\sum_{k=1}^K\l|\mathcal{A}^{(k)}\r|\l|\mathcal{S}^{(k)}\r|\cdot\sqrt{T}\r),\\
&G_{i,T}(d_0,\mathscr{P})\leq \mathcal{O}\l(m^{3/2}K^2\sum_{k=1}^K\l|\mathcal{A}^{(k)}\r|\l|\mathcal{S}^{(k)}\r|\cdot\sqrt{T}\r),~i=1,2,\cdots,m.
\end{align*}
\end{theorem}

\begin{proof}
Let $\Pi_*$ be the randomized stationary policy corresponding to the solution $\{\theta_*^{(k)}\}_{k=1}^K$ to \eqref{ori-lp1}-\eqref{ori-lp2} and let $\Pi$ be any randomized stationary policy such that $(d_0,\Pi)\in\mathcal{G}$. Since $G_{i,T}(d_{\Pi_*},\Pi_*) =  \sum_{t=0}^{T-1}\sum_{k=1}^K\l\langle\expect{\mathbf{g}_{i,t}},\theta^{(k)}_*\r\rangle\leq0$, it follows 
$(d_{\Pi_*},\Pi_*)\in\mathcal{G}$.
By Theorem \ref{thm:final-1}, we know that
\[
F_T(d_0,\mathscr{P}) - F_T(d_{\Pi_*},\Pi_*)\leq \mathcal{O}\l(m^{3/2}K^2\sum_{k=1}^K\l|\mathcal{A}^{(k)}\r|\l|\mathcal{S}^{(k)}\r|\cdot\sqrt{T}\r),
\]
and $G_{i,T}(d_0,\mathscr{P})$ satisfies the bound in the statement.
It is then enough to bound $F_T(d_{\Pi_*},\Pi_*) - F_T(d_{0},\Pi)$. We split it in to two terms:
\[
F_T(d_{\Pi_*},\Pi_*) - F_T(d_{0},\Pi)\leq \underbrace{F_T(d_{\Pi_*},\Pi_*) - F_T(d_{\Pi},\Pi)}_{\text{(I)}} + \underbrace{F_T(d_{\Pi},\Pi) - F_T(d_{0},\Pi)}_{\text{(II)}}.
\]
By \eqref{diff-1} in Lemma \ref{lemma:perturbed}, the term (II) is bounded by $C_1K\Psi$. It remains to bound the first term. Since $(d_0,\Pi)\in\mathcal{G}$, by Lemma \ref{lemma:perturbed}, the corresponding state-action probabilities $\{\theta^{(k)}\}_{k=1}^K$ of $\Pi$ satisfies $\sum_{k=1}^K\l\langle\expect{\mathbf{g}_{i,t}},\theta^{(k)}\r\rangle\leq C_1K\Psi/T$ and $\{\theta^{(k)}\}_{k=1}^K$ is feasible for \eqref{relax-lp1}-\eqref{relax-lp2}. Since $\{\overline{\theta}^{(k)}_*\}_{k=1}^K$ is the solution to \eqref{relax-lp1}-\eqref{relax-lp2}, we must have
\[ F_T(d_{\Pi},\Pi) = \sum_{t=0}^{T-1}\sum_{k=1}^K\l\langle \expect{\mathbf{f}_t^{(k)}}, \theta^{(k)}  \r\rangle 
\geq \sum_{t=0}^{T-1}\sum_{k=1}^K\l\langle \expect{\mathbf{f}_t^{(k)}}, \overline\theta_*^{(k)}  \r\rangle \]
On the other hand, by Corollary \ref{coro:comp},
\begin{align*}
\sum_{t=0}^{T-1}\sum_{k=1}^K\l\langle \expect{\mathbf{f}_t^{(k)}}, \overline\theta_*^{(k)}  \r\rangle
\geq& \sum_{t=0}^{T-1}\sum_{k=1}^K\l\langle \expect{\mathbf{f}^{(k)}}, \theta_*^{(k)}  \r\rangle - \frac{C_1K^2\sqrt{m}\Psi^2}{\eta}\\
=& F_T(d_{\Pi_*},\Pi_*) -  \frac{C_1K^2\sqrt{m}\Psi^2}{\eta}.
\end{align*}
Combining the above two displays gives $\text{(I)}\leq \frac{C_1K^2\sqrt{m}\Psi^2}{\eta}$ and the proof is finished.
\end{proof}


\bibliographystyle{unsrt}
\bibliography{OCMDP}

\appendix

\section{Additional proofs}

\subsection{Missing proofs in Section \ref{sec:prelim-thm}}
We prove Lemma \ref{lemma:mixing1} and \ref{lemma:prod-chain} in this section.

\begin{proof}[Proof of Lemma \ref{lemma:mixing1}]
For simplicity of notations, we drop the dependencies on $k$ throughout this proof.
We first show that for any $r\geq \widehat{r}$, where  $\widehat{r}$ is specified in Assumption \ref{assumption-1},
$
\mathbf{P}_{\pi_1}\mathbf{P}_{\pi_2}\cdots \mathbf{P}_{\pi_r}
$
is a strictly positive stochastic matrix.

Since the MDP is finite state with a finite action set, the set of all pure policies (Definition \ref{def:pp}) is finite. Let $\mathbf{P}_1,~\mathbf{P}_2,\cdots,~\mathbf{P}_N$ be probability transition matrices corresponding to these pure policies. Consider any sequence of randomized stationary policies $\pi_1,\cdots,\pi_r$. Then, it follows 
their transition matrices can be expressed as convex combinations of pure policies, i.e.
\[
\mathbf{P}_{\pi_1} = \sum_{i=1}^N\alpha^{(1)}_i\mathbf{P}_i,~~
\mathbf{P}_{\pi_2} = \sum_{i=1}^N\alpha^{(2)}_i\mathbf{P}_i,
~~\cdots,
\mathbf{P}_{\pi_r} = \sum_{i=1}^N\alpha^{(r)}_i\mathbf{P}_i,
\]
where $\sum_{i=1}^N\alpha^{(j)}_i = 1,~\forall j\in\{1,2,\cdots,r\}$ and $\alpha^{(j)}_i\geq0$.
Thus, we have the following display
\begin{align}
\mathbf{P}_{\pi_1}\mathbf{P}_{\pi_2}\cdots \mathbf{P}_{\pi_r}
=&\l( \sum_{i=1}^N\alpha^{(1)}_i\mathbf{P}_i\r)
\l( \sum_{i=1}^N\alpha^{(2)}_i\mathbf{P}_i\r)\cdots\l( \sum_{i=1}^N\alpha^{(r)}_i\mathbf{P}_i\r) \nonumber\\
=&\sum_{(i_1,\cdots,i_r)\in\mathcal{G}_r}
\alpha_{i_1}^{(1)}\cdots \alpha_{i_r}^{(r)}
\cdot     \mathbf{P}_{i_1}\mathbf{P}_{i_2}\cdots\mathbf{P}_{i_r},  \label{convex-comb}
\end{align}
where $\mathcal{G}_r$ ranges over all $N^r$ configurations.

Since $\l(\sum_{i=1}^N\alpha^{(1)}_i\r)\cdots\l(\sum_{i=1}^N\alpha^{(r)}_i\r) = 1$, it follows \eqref{convex-comb} is a convex combination of all possible sequences
$ \mathbf{P}_{i_1}\mathbf{P}_{i_2}\cdots\mathbf{P}_{i_r}$. By assumption \ref{assumption-1}, we have  $\mathbf{P}_{i_1}\mathbf{P}_{i_2}\cdots\mathbf{P}_{i_r}$ is strictly positive for any $(i_1,\cdots,i_r)\in\mathcal{G}_r$, and there exists a universal lower bound $\delta>0$ of all entries of $\mathbf{P}_{i_1}\mathbf{P}_{i_2}\cdots \mathbf{P}_{i_r}$ ranging over all configurations in $(i_1,\cdots,i_r)\in\mathcal{G}_r$.
This implies $\mathbf{P}_{\pi_1}\mathbf{P}_{\pi_2}\cdots \mathbf{P}_{\pi_r}$ is also strictly positive with the same lower bound $\delta>0$ for any sequences of randomized stationary policies $\pi_1,\cdots,\pi_r$.

Now, we proceed to prove the mixing bound. Choose $r=\widehat{r}$ and  we can
decompose any $\mathbf{P}_{\pi_1}\mathbf{P}_{\pi_2}\cdots \mathbf{P}_{\pi_r}$ as follows:
$$
\mathbf{P}_{\pi_1}\cdots \mathbf{P}_{\pi_r} = \delta\mathbf{\Pi} + (1-\delta)\mathbf{Q},
$$
where $\mathbf{\Pi}$ has each entry equal to $1/\l|\mathcal{S}\r|$ (recall that $\l|\mathcal{S}\r|$ is the number of states which equals the size of the matrix) and $\mathbf{Q}$ depends on $\pi_1,\cdots,\pi_r$. Then, $\mathbf{Q}$ is also a stochastic matrix (nonnegative and row sum up to 1) because both $\mathbf{P}_{\pi_1}\cdots \mathbf{P}_{\pi_r}$ and 
 $\mathbf{\Pi}$ are stochastic matrices. Thus, for any two distribution vectors $d_1$ and $d_2$, we have
 \begin{align*}
 \l(d_1-d_2\r)\mathbf{P}_{\pi_1}\cdots \mathbf{P}_{\pi_r} 
 =  \delta\l(d_1-d_2\r)\mathbf{\Pi} + (1-\delta)\l(d_1-d_2\r)\mathbf{Q}
 =(1-\delta)\l(d_1-d_2\r)\mathbf{Q},
 \end{align*}
 where we use the fact that for distribution vectors
 $$\l(d_1-d_2\r)\mathbf{\Pi} = \frac{1}{\l| \mathcal{S} \r|}\mathbf{1} - \frac{1}{\l| \mathcal{S}\r|}\mathbf{1} = 0.$$
Since $\mathbf{Q}$ is a stochastic matrix, it is non-expansive on $\ell_1$-norm, namely, for any vector $x$, $\|x\mathbf{Q}\|_1\leq \|x\|_1$. To see this, simply compute
\begin{multline}\label{non-expansive}
\|x\mathbf{Q}\|_1=\sum_{j=1}^{\l| \mathcal{S} \r|}\l| \sum_{i=1}^{\l| \mathcal{S} \r|}x_iQ_{ij} \r|
\leq\sum_{j=1}^{\l| \mathcal{S}\r|}  \sum_{i=1}^{\l| \mathcal{S}\r|}   \l|x_iQ_{ij} \r|
=\sum_{j=1}^{\l| \mathcal{S}\r|}  \sum_{i=1}^{\l| \mathcal{S}\r|}   \l|x_i \r|Q_{ij} = \sum_{i=1}^{\l| \mathcal{S}\r|}   \l|x_i \r| = \|x\|_1.
\end{multline}
Overall, we obtain,
\[
 \l\|\l(d_1-d_2\r)\mathbf{P}_{\pi_1}\cdots \mathbf{P}_{\pi_r} \r\|_1=  (1-\delta)\l\|\l(d_1-d_2\r)\mathbf{Q}\r\|_1
 \leq (1- \delta)\l\|d_1-d_2\r\|_1.
\]
We can then take $\tau = -\frac{1}{\log\l(1 - \delta\r)}$ to finish the proof. 
\end{proof}

\begin{proof}[Proof of Lemma \ref{lemma:prod-chain}]
Since the probability transition matrix of any randomized stationary policy is a convex combination of those of pure policies,
it is enough to show that the product MDP is irreducible and aperiodic under any joint pure policy.
For simplicity, let $\mathbf{s}_t = \l(s^{(1)},\cdots,s^{(K)}\r)$ and $\mathbf{a}_t = \l(a^{(1)},\cdots,a^{(K)}\r)$. Consider any joint pure policy $\Pi$ which select a fixed joint action 
$\mathbf{a}\in\mathcal{A}^{(1)}\times\cdots\times\mathcal{A}^{(K)}$ given a joint state $\mathbf{s}\in\mathcal{S}^{(1)}\times\cdots\times\mathcal{S}^{(K)}$, with probability 1.
By Assumption \ref{assumption:indep-trans}, we have
\begin{align}
&Pr\l(  s^{(1)}_{t+1},\cdots,s^{(K)}_{t+1}\l| s^{(1)}_{t},\cdots,s^{(K)}_{t},a^{(1)}_t,\cdots, a^{(K)}_t  \r.\r) \nonumber\\
=&Pr\l(  s^{(1)}_{t+1}\l| s^{(1)}_{t},\cdots,s^{(K)}_{t},a^{(1)}_t,\cdots, a^{(K)}_t, s^{(2)}_{t+1},\cdots,s^{(K)}_{t+1}  \r.\r)
 Pr\l(  s^{(2)}_{t+1},\cdots,s^{(K)}_{t+1}\l| s^{(1)}_{t},\cdots,s^{(K)}_{t},a^{(1)}_t,\cdots, a^{(K)}_t  \r.\r)  \nonumber\\
=&Pr\l(  s^{(1)}_{t+1}\l| s^{(1)}_{t},a^{(1)}_t  \r.\r)Pr\l(  s^{(2)}_{t+1},\cdots,s^{(K)}_{t+1}\l| s^{(1)}_{t},\cdots,s^{(K)}_{t},a^{(1)}_t,\cdots, a^{(K)}_t  \r.\r) \nonumber\\
=&\cdots=  \prod_{k=1}^{K-1}Pr\l(  s^{(k)}_{t+1}\l| s^{(k)}_{t},a^{(k)}_t  \r.\r)
\cdot Pr\l(  s^{(K)}_{t+1}\l| s^{(1)}_{t},\cdots,s^{(K)}_{t},a^{(1)}_t,\cdots, a^{(K)}_t  \r.\r)
= \prod_{k=1}^{K}Pr\l(  s^{(k)}_{t+1}\l| s^{(k)}_{t},a^{(k)}_t  \r.\r), \label{iter-expectation}
\end{align}
where the second equality follows from the independence relation in Assumption \ref{assumption:indep-trans}. Thus, we obtain the equality,
\[
Pr(\mathbf{s}_{t+1}=\mathbf{s}' \big| \mathbf{s}_t=\mathbf{s},\mathbf{a}_t=\mathbf{a})
= \prod_{k=1}^{K}Pr\l(  s^{(k)}_{t+1} = \tilde{s}^{(k)}\l| s^{(k)}_{t} = s^{(k)},a^{(k)}_t = a^{(k)}  \r.\r),
\]
Then, the one step transition probability between any two states
$\mathbf{s},\tilde{\mathbf{s}}\in\mathcal{S}^{(1)}\times\cdots\times\mathcal{S}^{(K)}$
can be computed as
\begin{align*}
&Pr(\mathbf{s}_{t+1}=\tilde{\mathbf{s}} \big| \mathbf{s}_t = \mathbf{s}) \\
=& \sum_{\mathbf{a}}Pr(\mathbf{s}_{t+1}=\tilde{\mathbf{s}} \big| \mathbf{s}_t=\mathbf{s},\mathbf{a}_t=\mathbf{a})
\cdot Pr(\mathbf{a}_t=\mathbf{a} \big|  \mathbf{s}_t=\mathbf{s})\\
=&\prod_{k=1}^{K}Pr\l(  s^{(k)}_{t+1} = \tilde{s}^{(k)}\l| s^{(k)}_{t} = s^{(k)},a^{(k)}_t = a^{(k)}  \r.\r)
\cdot Pr(\mathbf{a}_t=\mathbf{a} \big|  \mathbf{s}_t=\mathbf{s})\\
=&\prod_{k=1}^{K} P_{a^{(k)}(\mathbf{s})}\l(s^{(k)},\tilde{s}^{(k)}\r),
\end{align*}
where the notation $a^{(k)}(\mathbf{s})$ denotes a fixed mapping from product state space $\mathcal{S}^{(1)}\times\cdots\times\mathcal{S}^{(K)}$ to an individual action space 
$\mathcal{A}^{(k)}$ resulting from the pure policy,
$P_{a^{(k)}(\mathbf{s})}\l(s^{(k)},\tilde{s}^{(k)}\r)$ is the Markov transition probability from state $s^{(k)}$ to $\tilde{s}^{(k)}$ under the action $a^{(k)}(\mathbf{s})$. One can then further compute the $r$ ($r\geq2$) step transition probability from between any two states
$\mathbf{s},\tilde{\mathbf{s}}\in\mathcal{S}^{(1)}\times\cdots\times\mathcal{S}^{(K)}$ as
\begin{align}
&Pr(\mathbf{s}_{t+r}=\tilde{\mathbf{s}} \big| \mathbf{s}_t = \mathbf{s}) \nonumber\\
=&\sum_{\mathbf{s}_{t+r-1}}\cdots\sum_{\mathbf{s}_{t+1}}\prod_{k=1}^{K} P_{a^{(k)}(\mathbf{s})}\l(s^{(k)},s_{t+1}^{(k)}\r)\cdot
\prod_{k=1}^{K} P_{a^{(k)}(\mathbf{s}_{t+1})}\l(s_{t+1}^{(k)},s_{t+2}^{(k)}\r)\cdots
\prod_{k=1}^{K} P_{a^{(k)}(\mathbf{s}_{t+r-1})}\l(s_{t+r-1}^{(k)},\tilde{s}^{(k)}\r)     \nonumber\\
=&\sum_{\mathbf{s}_{t+r-1}}\cdots\sum_{\mathbf{s}_{t+1}}\prod_{k=1}^{K} P_{a^{(k)}(\mathbf{s})}\l(s^{(k)},s_{t+1}^{(k)}\r)\cdot
P_{a^{(k)}(\mathbf{s}_{t+1})}\l(s_{t+1}^{(k)},s_{t+2}^{(k)}\r)\cdots P_{a^{(k)}(\mathbf{s}_{t+r-1})}\l(s_{t+r-1}^{(k)},\tilde{s}^{(k)}\r). \label{lump-sum}
\end{align}

For any $k\in\{1,2,\cdots,K\}$, the term 
$$P_{a^{(k)}(\mathbf{s})}\l(s^{(k)},s_{t+1}^{(k)}\r)\cdot
P_{a^{(k)}(\mathbf{s}_{t+1})}\l(s_{t+1}^{(k)},s_{t+2}^{(k)}\r)\cdots P_{a^{(k)}(\mathbf{s}_{t+r-1})}\l(s_{t+r-1}^{(k)},\tilde{s}^{(k)}\r)$$
denotes the probability of moving from $s^{(k)}$ to $\tilde{s}^{(k)}$ along a certain path under a certain sequence of fixed decisions 
$a^{(k)}(\mathbf{s}),~a^{(k)}(\mathbf{s}_{t+1}),~\cdots,~a^{(k)}(\mathbf{s}_{t+r-1})$. 
Let $\mathbf{s}^{(k)} = \l( s_{t+1}^{(k)},s_{t+2}^{(k)},\cdots,s_{t+r-1}^{(k)} \r)\in\mathcal{S}^{(k)}\times\cdots\times\mathcal{S}^{(k)},~k\in\{1,2,\cdots,K\}$ be the state path of k-th MDP.
One can instead changing the order of summation in \eqref{lump-sum} and sum over state paths of each MDP as follows:
\begin{align*}
\eqref{lump-sum} &= \sum_{\mathbf{s}^{(K)}}\cdots\sum_{\mathbf{s}^{(1)}}
\prod_{k=1}^{K} P_{a^{(k)}(\mathbf{s})}\l(s^{(k)},s_{t+1}^{(k)}\r)\cdot
P_{a^{(k)}(\mathbf{s}_{t+1})}\l(s_{t+1}^{(k)},s_{t+2}^{(k)}\r)\cdots P_{a^{(k)}(\mathbf{s}_{t+r-1})}\l(s_{t+r-1}^{(k)},\tilde{s}^{(k)}\r)
\end{align*}
We would like to exchange the order of the product and the sums so that we can take the path sum over each individual MDP respectively. However, the problem is that the transition probabilities are coupled through the actions. 
The idea to proceed is to first apply a ``hard'' decoupling by taking the infimum of transition probabilities of each MDP over all pure policies, and use Assumption \ref{assumption-1}, to bound the transition probability from below uniformly. We have
\begin{align*}
\eqref{lump-sum}\geq
&\inf_{\mathbf{s}^{(1)}}\sum_{\mathbf{s}^{(K)}}\cdots\sum_{\mathbf{s}^{(2)}}\prod_{k=2}^{K}P_{a^{(k)}(\mathbf{s})}\l(s^{(k)},s_{t+1}^{(k)}\r)\cdots P_{a^{(k)}(\mathbf{s}_{t+r-1})}\l(s_{t+r-1}^{(k)},\tilde{s}^{(k)}\r)\\
&\cdot\inf_{\mathbf{s}^{(j)},~j\neq1}\sum_{\mathbf{s}^{(1)}} P_{a^{(1)}(\mathbf{s})}\l(s^{(1)},s_{t+1}^{(1)}\r)\cdots P_{a^{(1)}(\mathbf{s}_{t+r-1})}\l(s_{t+r-1}^{(1)},\tilde{s}^{(1)}\r)\\
\geq&\inf_{\mathbf{s}^{(1)}}\sum_{\mathbf{s}^{(K)}}\cdots\sum_{\mathbf{s}^{(2)}}\prod_{k=2}^{K}P_{a^{(k)}(\mathbf{s})}\l(s^{(k)},s_{t+1}^{(k)}\r)\cdots P_{a^{(k)}(\mathbf{s}_{t+r-1})}\l(s_{t+r-1}^{(k)},\tilde{s}^{(k)}\r)\\
&\cdot\inf_{\pi_1^{(1)},\cdots,\pi_r^{(1)}}\sum_{\mathbf{s}^{(1)}} P_{\pi_1^{(1)}}\l(s^{(1)},s_{t+1}^{(1)}\r)\cdots P_{\pi_r^{(1)}}\l(s_{t+r-1}^{(1)},\tilde{s}^{(1)}\r),
\end{align*}
where $\pi_1^{(1)},\cdots,\pi_r^{(1)}$ range over all pure policies, and the second inequality follows from the fact that \textit{fix any path of other MDPs}  (i.e. $\mathbf{s}^{(j)},~j\neq1$), the term 
$$\sum_{\mathbf{s}^{(1)}} P_{a^{(1)}(\mathbf{s})}\l(s^{(1)},s_{t+1}^{(1)}\r)\cdots P_{a^{(1)}(\mathbf{s}_{t+r-1})}\l(s_{t+r-1}^{(k)},\tilde{s}^{(1)}\r)$$
is the probability of reaching $\tilde{s}^{(1)}$ from $s^{(1)}$ in $r$ steps using a sequence of actions
$a^{(1)}(\mathbf{s}^{(1)}),\cdots,a^{(1)}(\mathbf{s}^{(1)}_{t+r-1})$, where each action is a deterministic function of the previous state at the 1-st MDP only. Thus, it dominates the infimum over all sequences of pure policies $\pi_1^{(1)},\cdots,\pi_r^{(1)}$ on this MDP. Similarly, we can decouple the rest of the sums and obtain the follow display:
\begin{align*}
\eqref{lump-sum}\geq&
\prod_{k=1}^K\inf_{\pi_1^{(k)},\cdots,\pi_r^{(k)}}\sum_{\mathbf{s}^{(k)}} P_{\pi_1^{(k)}}\l(s^{(k)},s_{t+1}^{(k)}\r)\cdots P_{\pi_r^{(k)}}\l(s_{t+r-1}^{(k)},\tilde{s}^{(k)}\r)\\
=&\prod_{k=1}^K\inf_{\pi_1^{(k)},\cdots,\pi_r^{(k)}}P_{\pi_1^{(k)},\cdots,\pi_r^{(k)}}\l(s^{(k)}, \tilde{s}^{(k)}\r),
\end{align*}
where $P_{\pi_1^{(k)},\cdots,\pi_r^{(k)}}\l(s^{(k)}, \tilde{s}^{(k)}\r)$ denotes the $\l(s^{(k)}, \tilde{s}^{(k)}\r)$-th entry of the product matrix 
$\mathbf{P}_{\pi_1^{(k)}}^{(k)}\cdots \mathbf{P}_{\pi_r^{(k)}}^{(k)}$.
Now, by Assumption \ref{assumption-1}, there exists a large enough integer $\widehat{r}$ such that $\mathbf{P}_{\pi_1^{(k)}}^{(k)}\cdots \mathbf{P}_{\pi_r^{(k)}}^{(k)}$ is a strictly positive matrix for any sequence of $r\geq \widehat{r}$ randomized stationary policy. As a consequence, the above probability is strictly positive and \eqref{lump-sum} is also strictly positive.

This implies, if we choose $\tilde{\mathbf{s}} = \mathbf{s}$, then, starting from any arbitrary product state 
$\mathbf{s}\in\mathcal{S}^{(1)}\times\cdots\times\mathcal{S}^{(K)}$, there is a positive probability of returning to this state after $r$ steps for all $r\geq\widehat{r}$, which gives the aperiodicity. Similarly, there is a positive probability of reaching any other composite state after $r$ steps for all $r\geq\widehat{r}$, which gives the irreducibility. This implies the product state MDP is irreducible and aperiodic under any joint pure policy, and thus, any joint randomized stationary policy.

For the second part of the claim, we consider any randomized stationary policy $\Pi$ and the corresponding joint transition probability matrix $\mathbf{P}_\Pi$, there exists a stationary state-action probability vector $\Phi(\mathbf{a},\mathbf{s}),~\mathbf{a}\in\mathcal{A}^{(1)}\times\cdots\times\mathcal{A}^{(K)},~\mathbf{s}\in\mathcal{S}^{(1)}\times\cdots\times\mathcal{S}^{(K)}$, such that
\begin{equation}\label{inter-1}
\sum_{\mathbf{a}}\Phi(\mathbf{a},\tilde{\mathbf{s}}) = \sum_{\mathbf{s}}\sum_{\mathbf{a}}\Phi(\mathbf{a},\mathbf{s})P_\mathbf{a}(\mathbf{s},\tilde{\mathbf{s}}),
~\forall \tilde{\mathbf{s}}\in \mathcal{S}^{(1)}\times\cdots\times\mathcal{S}^{(K)}.
\end{equation}
Then, the state-action probability of the k-th MDP is $\theta^{(k)}(a^{(k)},\tilde{s}^{(k)}) =\sum_{\tilde{s}^{(j)},a^{(j)},~j\neq k}\Phi(\mathbf{a},\tilde{\mathbf{s}})$. Thus,
\begin{align*}
\sum_{a^{(k)}}\theta^{(k)}(a^{(k)},\tilde{s}^{(k)})
=&\sum_{\tilde{s}^{(j)},~j\neq k}\sum_{\mathbf{a}}\Phi(\mathbf{a},\tilde{\mathbf{s}})
= \sum_{\mathbf{s}}\sum_{\mathbf{a}}\Phi(\mathbf{a},\mathbf{s})\sum_{\tilde{s}^{(j)},~j\neq k}P_\mathbf{a}(\mathbf{s},\tilde{\mathbf{s}})\\
=& \sum_{\mathbf{s}}\sum_{\mathbf{a}}\Phi(\mathbf{a},\mathbf{s})\cdot Pr\l(\tilde{s}^{(k)}|\mathbf{a},\mathbf{s}\r)
= \sum_{\mathbf{s}}\sum_{\mathbf{a}}\Phi(\mathbf{a},\mathbf{s})\cdot Pr\l(\tilde{s}^{(k)}| a^{(k)}, s^{(k)}\r)\\
=&\sum_{a^{(k)}}\sum_{s^{(k)}}\theta^{(k)}(a^{(k)},\tilde{s}^{(k)})\cdot Pr\l(\tilde{s}^{(k)}| a^{(k)}, s^{(k)}\r)\\
=&\sum_{a^{(k)}}\sum_{s^{(k)}}\theta^{(k)}(a^{(k)},\tilde{s}^{(k)})\cdot P_{a^{(k)}}\l( s^{(k)},\tilde{s}^{(k)}\r)
\end{align*}
where the second from the last inequality follows from Assumption \ref{assumption:indep-trans}. This finishes the proof.
\end{proof}

\subsection{Missing proofs in Section \ref{sec:stat-analysis}}
\begin{proof}[Proof of Lemma \ref{lemma:queue-drift-bound}]
Consider the state-action probabilities $\{\tilde\theta^{(k)}\}_{k=1}^K$ which achieves the Slater's condition in \eqref{slater-2}. First of all, note that $Q_i(t)\in\mathcal{F}_{t-1},~\forall t\geq1$. 
Then, using the Assumption \ref{assumption:oblivious} that $\{\mathbf{g}_{i,t-1}^{(k)}\}_{k=1}^K$ is independent of all system information up to $t-1$, we have 
\begin{equation}\label{neg-drift}
\expect{Q_i(t-1)\sum_{k=1}^K\l\langle \mathbf{g}_{i,t-1}^{(k)}, \tilde\theta \r\rangle \Big|~\mathcal{F}_{t-1}}
=\expect{\sum_{k=1}^K\l\langle \mathbf{g}_{i,t-1}^{(k)}, \tilde\theta \r\rangle}Q_i(t-1)\leq -\eta Q_i(t-1).
\end{equation}
Now, by the drift-plus-penalty bound \eqref{dpp-bound}, with $\theta^{(k)}=\tilde\theta^{(k)}$,
\begin{align*}
\Delta(t)\leq& - V\sum_{k=1}^{K}\l\langle \mathbf{f}_{t-1}^{(k)}, \theta_t^{(k)} - \theta_{t-1}^{(k)} \r\rangle - \alpha\sum_{k=1}^K\|\theta^{(k)}_t-\theta^{(k)}_{t-1}\|_2^2+ \frac32mK^2\Psi^2
+ V\sum_{k=1}^K\l\langle \mathbf{f}_{t-1}^{(k)}, \tilde\theta^{(k)} - \theta_{t-1}^{(k)} \r\rangle \\
 &+ \sum_{i=1}^mQ_i(t-1)\sum_{k=1}^K
\l\langle\mathbf{g}_{i,t-1}^{(k)}, \tilde\theta^{(k)}\r\rangle + \alpha\sum_{k=1}^K\|\tilde\theta^{(k)} - \theta_{t-1}^{(k)}\|_2^2
-\alpha\sum_{k=1}^K\|  \tilde\theta^{(k)}-\theta_t^{(k)}\|_2^2\\
\leq& 4VK\Psi  +  \frac32mK^2\Psi^2  +  \sum_{i=1}^mQ_i(t-1)\sum_{k=1}^K
\l\langle\mathbf{g}_{i,t-1}^{(k)}, \tilde\theta^{(k)}\r\rangle + \alpha\sum_{k=1}^K\|\tilde\theta^{(k)} - \theta_{t-1}^{(k)}\|_2^2
-\alpha\sum_{k=1}^K\|  \tilde\theta^{(k)}-\theta_t^{(k)}\|_2^2
\end{align*}
where the second inequality follows from Holder's inequality that 
$$\l|\l\langle \mathbf{f}_{t-1}^{(k)}, \theta_t^{(k)} - \theta_{t-1}^{(k)} \r\rangle\r|\leq \|\mathbf{f}_{t-1}^{(k)}\|_{\infty}\l\| \theta_t^{(k)} - \theta_{t-1}^{(k)}\r\|_1\leq 2\Psi.$$
Summing up the drift from $t$ to $t+t_0-1$ and taking a conditional expectation $\expect{\cdot|\mathcal{F}_{t-1}}$ give
\begin{multline*}
\expect{\|\mathbf{Q}(t+t_0)\|_2^2-\|\mathbf{Q}(t)\|_2^2 \Big| \mathcal{F}_{t-1}}
\leq 8VK\Psi  +  3mK^2\Psi^2 + 2\sum_{i=1}^m\expect{\sum_{\tau=t}^{t+t_0-1}Q_i(\tau-1)\sum_{k=1}^K
\l\langle\mathbf{g}_{i,\tau-1}^{(k)}, \tilde\theta^{(k)}\r\rangle \Big| \mathcal{F}_{t-1}}\\
+2\alpha\expect{\sum_{k=1}^K\l(\|\tilde\theta^{(k)} - \theta_{t-1}^{(k)}\|_2^2  -  \|  \tilde\theta^{(k)}-\theta_{t+t_0}^{(k)}\|_2^2\r) \Big|\mathcal{F}_{t-1}}\\
\leq 8VK\Psi  +  3mK^2\Psi^2 + 4K\alpha + 2\sum_{i=1}^m\expect{\sum_{\tau=t}^{t+t_0-1}Q_i(\tau-1)\sum_{k=1}^K
\l\langle\mathbf{g}_{i,\tau-1}^{(k)}, \tilde\theta^{(k)}\r\rangle \Big| \mathcal{F}_{t-1}}.
\end{multline*}
Using the tower property of conditional expectations (further taking conditional expectations 
$\expect{\cdot\Big|\mathcal{F}_{t+t_0-1}\cdots\Big|\mathcal{F}_t}$ inside the conditional expectation) and the bound \eqref{neg-drift}, we have 
\begin{multline*}
\expect{\sum_{\tau=t}^{t+t_0-1}Q_i(\tau-1)\sum_{k=1}^K
\l\langle\mathbf{g}_{i,\tau-1}^{(k)}, \tilde\theta^{(k)}\r\rangle \Big| \mathcal{F}_{t-1}}\leq-\eta\expect{\sum_{\tau=t}^{t+t_0-1}Q_i(\tau-1)\Big| \mathcal{F}_{t-1}}\\
\leq -\eta t_0 Q_i(t-1) + \frac{t_0(t_0-1)}{2}\Psi\leq -\eta t_0 Q_i(t) + \frac{t_0(t_0-1)}{2}\Psi + \eta t_0K\Psi,
\end{multline*}
where the last inequality follows from the queue updating rule \eqref{Q-update} that 
$$|Q_i(t-1)-Q_i(t)|\leq \l| \sum_{k=1}^K\l\langle \mathbf{g}_{i,t-2}^{(k)},\theta_{t-1}^{(k)} \r\rangle \r|
 \leq K\|\mathbf{g}_{i,t-2}^{(k)}\|_\infty \|\theta_{t-1}^{(k)}\|_1 \leq K\Psi.$$
Thus, we have 
\begin{multline*}
\expect{\|\mathbf{Q}(t+t_0)\|_2^2-\|\mathbf{Q}(t)\|_2^2 \Big| \mathcal{F}_{t-1}} 
\leq  8VK\Psi  +  3mK^2\Psi^2 + 4K\alpha + t_0(t_0-1)m\Psi \\
 + 2mK\Psi\eta t_0 - 2 \eta t_0 \sum_{i=1}^mQ_i(t)\\
 \leq8VK\Psi  +  3mK^2\Psi^2 + 4K\alpha + t_0(t_0-1)m\Psi 
 + 2mK\Psi\eta t_0 - 2 \eta t_0 \|Q_i(t)\|_2.
\end{multline*}
Suppose $\|Q_i(t)\|_2\geq\frac{8VK\Psi  +  3mK^2\Psi^2 + 4K\alpha + t_0(t_0-1)m\Psi 
 + 2mK\Psi\eta t_0 + \eta^2 t_0^2}{\eta t_0}$, then, it follows,
 \[
 \expect{\|\mathbf{Q}(t+t_0)\|_2^2-\|\mathbf{Q}(t)\|_2^2 \Big| \mathcal{F}_{t-1}} \leq -  \eta t_0 \|Q_i(t)\|_2  ,
 \]
which implies 
\[
 \expect{\|\mathbf{Q}(t+t_0)\|_2^2 \Big| \mathcal{F}_{t-1}} \leq \l(\|Q_i(t)\|_2-\frac{\eta t_0}{2}\r)^2
\]
Since $\|Q_i(t)\|_2\geq \frac{\eta t_0}{2}$, taking square root from both sides using Jensen' inequality gives
\[
 \expect{\|\mathbf{Q}(t+t_0)\|_2 \Big| \mathcal{F}_{t-1}}\leq \|Q_i(t)\|_2 - \frac{\eta t_0}{2}.
\]
On the other hand, we always have
\begin{multline*}
\Big|\|\mathbf{Q}(t+1)\|_2-\|\mathbf{Q}(t)\|_2\Big|=\l|\sqrt{\sum_{i=1}^m\max\l\{Q_i(t)+\sum_{k=1}^{K}\l\langle\mathbf{g}_{i,t-1}^{(k)},\theta_{t}^{(k)} \r\rangle,0\r\}^2}
- \sqrt{\sum_{i=1}^mQ_i(t)^2}\r|\\
\leq \l(\sum_{i=1}^m\l(\sum_{k=1}^{K}\l\langle\mathbf{g}_{i,t-1}^{(k)},\theta_{t}^{(k)} \r\rangle\r)^2\r)^{1/2}\leq\sqrt{m}K\Psi.
\end{multline*}
Overall, we finish the proof.
\end{proof}

\subsection{Missing proofs in Section \ref{sec:perturb}}
\begin{proof}[Proof of Lemma \ref{lemma:small-bound}]
Consider any joint randomized stationary policy $\Pi$ and a starting state probability $d_0$ on the product state space $\mathcal{S}^{(1)}\times\mathcal{S}^{(2)}\times\cdots\times\mathcal{S}^{(K)}$. Let $\mathbf{P}_\Pi$ be the corresponding transition matrix on the product state space. Let $d_t$ be the state distribution at time $t$ under $\Pi$ and $d_\Pi$ be the stationary state distribution. 
By Lemma \ref{lemma:prod-chain}, we know that this product state MDP is irreducible and aperiodic (ergodic) under any randomized stationary policy. In particular, it is ergodic under any pure policy. Since there are only finitely many pure policies, let $\mathbf{P}_{\Pi_1},\cdots,\mathbf{P}_{\Pi_N}$ be probability transition matrices corresponding to these pure policies. By Proposition 1.7 of \cite{LevinPeresWilmer2006}
, for any $\Pi_i,~i\in\{1,2,\cdots,N\}$, there exists integer $\tau_i>0$ such that 
$\l(\mathbf{P}_{\Pi_i}\r)^t$ is strictly positive for any $t\geq\tau_i$. Let 
$$\tau_1=\max_i\tau_i,$$ 
then, it follows $\l(\mathbf{P}_{\Pi_i}\r)^{\tau_1}$ is strictly positive uniformly for all $\Pi_i$'s.
Let $\delta>0$ be the least entry of $\l(\mathbf{P}_{\Pi_i}\r)^{\tau_1}$ over all $\Pi_i$'s.
Following from the fact that the probability transition matrix $\mathbf{P}_{\Pi}$ is a convex combination of those of pure policies, i.e. $\mathbf{P}_{\Pi}=\sum_{i=1}^N\alpha_i \mathbf{P}_{\Pi_i},~\alpha_i\geq0,~\sum_{i=1}^N\alpha_i=1$, we have $\l(\mathbf{P}_{\Pi}\r)^{\tau_1}$ is also strictly positive. To see this, note that
\[
\l(\mathbf{P}_{\Pi}\r)^{\tau_1} = \l(\sum_{i=1}^N\alpha_i \mathbf{P}_{\Pi_i}\r)^{\tau_1}\geq \sum_{i=1}^N\alpha_i^{\tau_1} \l(\mathbf{P}_{\Pi_i}\r)^{\tau_1}>0,
\]
where the inequality is taken to be entry-wise. Furthermore, the least entry of $\l(\mathbf{P}_{\Pi}\r)^{\tau_1}$ is lower bounded by $\delta/N^{\tau_1-1}$ uniformly over all joint randomized stationary policies $\Pi$, which follows from the fact that the least entry of $\frac{1}{N}\l(\mathbf{P}_{\Pi}\r)^{\tau_1}$ is bounded as
\[
\frac{1}{N}\sum_{i=1}^N\alpha_i^{\tau_1}\delta \geq \l(\frac{1}{N}\sum_{i=1}^N\alpha_i\r)^{\tau_1}\delta= \frac{\delta}{N^{\tau_1}}.
\]
The rest is a standard bookkeeping argument following from the Markov chain mixing time theory (Theorem 4.9 of \cite{LevinPeresWilmer2006}). Let $\mathbf{D}_\Pi$ be a matrix of the same size as $\mathbf{P}_{\Pi}$ and each row equal to the stationary distribution $d_\Pi$. Let 
$\varepsilon = \delta/N^{\tau_1-1}$. We claim that for any 
integer $n>0$, and any $\Pi$,
\begin{equation}\label{small-fact-1}
\mathbf{P}_{\Pi}^{\tau_1n} = (1-(1-\varepsilon)^{n})\mathbf{D}_\Pi + (1-\varepsilon)^{n}\mathbf{Q}^n,
\end{equation}
for some stochastic matrix $\mathbf{Q}$. We use induction to prove this claim. First of all, for $n=1$, from the fact that $\l(\mathbf{P}_{\Pi}\r)^{\tau_1}$ is a positive matrix and the least entry is uniformly lower bounded by $\varepsilon$ over all policies $\Pi$, we can write $\l(\mathbf{P}_{\Pi}\r)^{\tau_1}$ as
\[
 \l(\mathbf{P}_{\Pi}\r)^{\tau_1} = \varepsilon\mathbf{D}_\Pi + (1-\varepsilon)\mathbf{Q},
\]
for some stochastic matrix $\mathbf{Q}$, where we use the fact that $\varepsilon\in(0,1]$. Suppose \eqref{small-fact-1} holds for $n=1,2,\cdots,\ell$, we show that it also holds for $n=\ell+1$. Using the fact that $\mathbf{D}_\Pi\mathbf{P}_\Pi = \mathbf{D}_\Pi$ and $\mathbf{Q}\mathbf{D}_\Pi = \mathbf{D}_\Pi$ for any stochastic matrix $\mathbf{Q}$,  we can write out $\mathbf{P}_{\Pi}^{\tau_1(\ell+1)}$:
\begin{align*}
\mathbf{P}_{\Pi}^{\tau_1(\ell+1)} =& \mathbf{P}_{\Pi}^{\tau_1\ell}\mathbf{P}_{\Pi}^{\tau_1} = \l( \l(1 - (1-\varepsilon)^\ell\r)\mathbf{D}_{\Pi} + (1-\varepsilon)^\ell Q^{\ell} \r)\mathbf{P}_{\Pi}^{\tau_1}\\
=&  \l(1 - (1-\varepsilon)^\ell\r)\mathbf{D}_{\Pi}\mathbf{P}_{\Pi}^{\tau_1} + (1-\varepsilon)^\ell \mathbf{Q}^{\ell}\mathbf{P}_{\Pi}^{\tau_1}\\
=&  \l(1 - (1-\varepsilon)^\ell\r)\mathbf{D}_{\Pi} + (1-\varepsilon)^\ell \mathbf{Q}^{\ell}(\varepsilon\mathbf{D}_\Pi + (1-\varepsilon)\mathbf{Q})\\
=&  \l(1 - (1-\varepsilon)^\ell\r)\mathbf{D}_{\Pi} + (1-\varepsilon)^\ell \mathbf{Q}^{\ell}((1-(1-\varepsilon))\mathbf{D}_\Pi + (1-\varepsilon)\mathbf{Q})\\
= & (1-(1-\varepsilon)^{\ell+1})\mathbf{D}_\Pi + (1-\varepsilon)^{\ell+1}\mathbf{Q}^{\ell+1}.
\end{align*}
Thus, \eqref{small-fact-1} holds. For any integer $t>0$, we write $t=\tau_1n+j$ for some integer $j\in[0,\tau_1)$ and $n\geq0$. Then,
\begin{equation*}
\l(\mathbf{P}_{\Pi}\r)^{t}-\mathbf{D}_\Pi  = \l(\mathbf{P}_{\Pi}\r)^{t}-\mathbf{D}_\Pi = (1-\varepsilon)^n\l(Q^n\mathbf{P}_{\Pi}^j-\mathbf{D}_\Pi\r).
\end{equation*} 
Let $\mathbf{P}_{\Pi}^{t}(i,\cdot)$ be the $i$-th row of $\mathbf{P}_{\Pi}^{t}$, then, we obtain
\[
\max_{i}\|  \mathbf{P}_{\Pi}^{t}(i,\cdot)  -  d_\Pi  \|_1\leq 2(1-\varepsilon)^n, 
\]
where we use the fact that the $\ell_1$-norm of the row difference is bounded by 2. Finally, for any starting state distribution $d_0$, we have
\begin{align*}
\l\|d_0\mathbf{P}_{\Pi}^t - d_{\Pi}  \r\|_1 =& \l\| \sum_{i} d_0(i)\l( \mathbf{P}_{\Pi}^t(i,\cdot) - d_\Pi \r) \r\|_1\\
=&  \sum_{i} d_0(i)\l\| \mathbf{P}_{\Pi}^t(i,\cdot) - d_\Pi \r\|_1\leq \max_{i}\|  \mathbf{P}_{\Pi}^{t}(i,\cdot)  -  d_\Pi  \|_1\leq 2(1-\varepsilon)^n.
\end{align*}
Take $r_1 = \log\frac{1}{1-\varepsilon}$
finishes the proof.
\end{proof}

\end{document}